\documentclass{article}
\usepackage{latexsym,amsfonts,amsmath,amsthm,makeidx,color}
\usepackage{graphicx}
\usepackage{makeidx}
\usepackage{float}
\usepackage{subfigure}
\makeindex

\usepackage{geometry}
\allowdisplaybreaks

\newtheorem{thm}{Theorem}[section]
\newtheorem{lemma}[thm]{Lemma}

\newtheorem{remark}[thm]{Remark}

\numberwithin{equation}{section}
\allowdisplaybreaks

\newcommand{\al}{\alpha}
\newcommand{\be}{\beta}
\newcommand{\ga}{\gamma}

\newcommand{\De}{\Delta}

\newcommand{\Om}{\Omega}

\newcommand{\Si}{\Sigma}
\newcommand{\la}{\lambda}

\newcommand{\ti}{\widetilde}

\newcommand{\R}{{\mathbb R}}

\begin{document}

\bibliographystyle{amsplain}

\title{\bf Normalized solutions to Sch\"odinger equations with  potential and inhomogeneous nonlinearities on large convex domains}

\author{\bf Thomas Bartsch\;\&\; \bf Shijie Qi\;\&\; Wenming Zou\footnote{Supported by NSFC  (No.12171265,  11025106, 12201482).} \\
}

\date{ }
\maketitle


\begin{abstract}\noindent
The paper addresses an open problem raised in \cite{Bartsch-Molle-Rizzi-Verzini2021} on the existence of normalized solutions to Schr\"odinger equations with potentials and inhomogeneous nonlinearities. We consider the problem 
\[
  -\De u+V(x)u+\la u = |u|^{q-2}u+\be |u|^{p-2}u, \quad \|u\|^2_2=\int|u|^2dx = \al,
\]
both on $\R^N$ as well as on domains $r\Om$ where $\Om\subset\R^N$ is an open bounded convex domain and $r>0$ is large. The exponents satisfy $2<p<2+\frac4N<q<2^*=\frac{2N}{N-2}$, so that the nonlinearity is a combination of a mass subcritical and a mass supercritical term. Due to the presence of the potential a by now standard approach based on the Pohozaev identity cannot be used. We develop a robust method to study the existence of normalized solutions of nonlinear Schr\"odinger equations with potential and find conditions on $V$ so that normalized solutions exist. Our results are new even in the case $\be=0$.
\vspace{0.3cm}
\vspace{0.3cm}
\end{abstract}

\section{Introduction and main results}
The nonlinear Schr\"odinger equation
\begin{equation}\label{schr}
-i\frac{\partial \Phi}{\partial t}=\De \Phi -V(x)\Phi+f(|\Phi|)\Phi\quad\text{in}\; \R\times\R^N
\end{equation}
for a wave function $\Phi(t,x)\in\mathbb C$ is a fundamental equation in quantum mechanics, used to model, for instance, nonlinear optical problems or Bose-Einstein condensates; see e.g.\ \cite{Choi-Niu,Erdos-Schlein-Yau,Gross}. Since the remarkable experiments on Bose-Einstein condensates in dilutes gases of alkali atoms in 1995 \cite{Anderson-Ensher-Matthews-Wieman-Cornell} several types of potentials have been used, see e.g., \cite{Choi-Niu,Gupta-Murch,Ryu-Andersen}. Consequently, nonlinear Schr\"odinger equations have been a central topic in mathematics and theoretical physics during the last  decades.

\vskip 0.1in
In particular, many papers focus on standing wave solutions of \eqref{schr}, that is, solutions of the form
\begin{equation}
	\label{standing}
	\Phi(t, x)=e^{-i \la t} u(x).
\end{equation}
This ansatz leads to the equation
\begin{equation}\label{transchr}
  -\De u+V(x)u+\la u= f(|u|)u\quad\text{in}\; \R^N
\end{equation}
for the unknowns $u\in H^1(\R^N)$, $\la\in\R$. There are two different points of view concerning the ansatz \eqref{standing}: one can either prescribe the frequency $\la$ or the $L^2$ norm $\|u\|_2$. Recall that the $L^2$ norm of $\Phi(t,\cdot)$ is conserved in time for a solution $\Phi$ of \eqref{schr}. In spite of the clear physical relevance of solutions with prescribed $L^2$ norm their existence and properties are much less understood compared with solutions of \eqref{transchr} for prescribed $\la$. 
\vskip 0.1in

Normalized solutions correspond to critical points of the functional $E_V:H^1(\R^N)\to \R$ defined by
\begin{equation}\label{eq:E_V}
  E_V(u) = \frac{1}{2}\int_{\R^N}|\nabla u|^2dx + \frac{1}{2}\int_{\R^N}Vu^2dx - \int_{\R^N}F(u)dx
\end{equation}
on the mass constraint manifold
\begin{equation}\label{mass}
  S_\al=\left\{u\in H^1(\R^N): \|u\|_2^2=\al\right\},
\end{equation}
where $F(s)=\int_0^sf(|t|)t\,dt$. In fact, a critical point $u\in S_\al$ of $E_V|_{S_\al}$ is a solution \eqref{transchr} with $\la$ the corresponding Lagrange multiplier. The difficulty of this approach is the lack of the Palais-Smale condition, even if the nonlinearity is Sobolev subcritical. This is different from the variational approach corresponding to the fixed frequency problem. 
\vskip 0.1in

In the autonomous case, i.e.\ the potential $V$ is constant, an approach based on the Pohozaev identity and going back to Jeanjean's paper \cite{Jeanjean1997}, has been used successfully in recent years. The main ingredient of this approach is the use of the transformation $s*u(x)= e^{sN/2}u(e^sx)$ in order to find bounded Palais-Smale sequences. Recall that the identity $P(u):=\frac{d}{ds}E_V(s*u)=0$ is the Pohozaev identity. In \cite{Bartsch-Soave2017} it is proved, in the case $V=0$, that the Pohozaev manifold $\{u\in S_\al:P(u)=0\}$ can be used as natural constraint for the problem even though it may only be a topological manifold, depending on $f$. Jeanjean and Lu \cite{Jeanjean-Lu2020} and Bieganowski and Mederski \cite{Bieganowski-Mederski2021} improved these results by weakening the conditions in \cite{Jeanjean1997,Bartsch-Soave2017}. These papers are concerned with mass supercritical nonlinearities; a typical example is $f(|u|)u=|u|^{p-2}u$, $2+\frac4N<p<2^*$, or sums of such nonlinearities. Recently, Soave \cite{Soave2020-I,Soave2020-II} studied normalized solution of \eqref{transchr} with combined nonlinearity $f(|u|)u=|u|^{q-2}u+\be |u|^{p-2}u$, $2<p<2+\frac4N<q\le2^*$, by decomposing the Nehari-Pohozaev manifold in a subtle way using the more complicated geometric structure of the energy functional. We would also like to mention the papers \cite{Jeanjean-Le2021-I,Wei-Wu} for the multiplicity of normalized solutions in the Sobolev critical case, \cite{Bartsch-Jeanjean-Soave2016, Bartsch-Soave2017, Bartsch-Soave2019, Mederski-Schino2022} and references therein for systems of nonlinear Schr\"odinger equations, and \cite{Bellazzini-Jeanjean-Luo2013, Bonheure-Casteras-Gou-Jeanjean2019-I, Qi-Zou} and the references therein for the study of normalized solutions to other autonomous equations.
\vskip 0.1in

The methods of the papers mentioned above depend heavily on $V$ being constant and do not work for non-constant $V$, even in the case of a radial potential $V(x)=V(|x|)$. They also do not work on bounded domains. In the presence of a potential $V$ the existence of a normalized solution of \eqref{transchr} has been obtained recently by Ikoma and Miyamoto \cite{Ikoma-Miyamoto2020} provided $V\le0$, $V(x)\to0$ as $|x|\to\infty$ and $V$ satisfies additional technical conditions, and $f$ is mass subcritical, i.e.\ $f(|u|)u=o(|u|^{1+4/N})$ as $|u|\to\infty$. The conditions on $V$ have been considerably relaxed by Yang et al.\ \cite{Yang-Qi-Zou}. In the mass subcritical case one can try to minimize $E_V$ on $S_\al$ as in the seminal paper \cite{Lions1984} of P.L.\ Lions. For mass supercritical $f$ so far only the case of a homogeneous nonlinearity $f(|u|)u=|u|^{p-2}u$ has been considered. Noris et al. \cite{Noris-Tavares-Verzini2015} considered confining potentials where $V(x)\to\infty$ as $|x|\to\infty$, Bellazzini et al.\ \cite{Bellazzini-Boussaid-Jeanjean-Visciglia2017} treated partially confining potentials, and Bartsch et al.\ \cite{Bartsch-Molle-Rizzi-Verzini2021} decaying potentials, i.e.\ $V(x)\to0$ as $|x|\to\infty$. The existence of normalized solutions on bounded domains has been obtained in \cite{Noris-Tavares-Terracini-Verzini2012, Noris-Tavares-Verzini2014, Noris-Tavares-Verzini2019, Pierotti-Verzini2017}. In all these papers the homogeneity of the nonlinearity was essential. As mentioned in \cite{Bartsch-Molle-Rizzi-Verzini2021} the existence of normalized solutions of \eqref{transchr} is an open problem when $V$ is not constant and $f$ is mass supercritical and not homogeneous. But even in the autonomous case, the existence of normalized solutions of nonlinear Schr\"odinger equations with combined nonlinearities in bounded domains still remains open.
\vskip 0.1in

In the present paper, we address these open problems and study the existence and multiplicity of solutions of
\begin{equation}\label{eq}
\begin{cases}
  -\De u+V(x)u+\la u=|u|^{q-2}u+\be |u|^{p-2}u\quad \text{in}\  \Om,\\
  u\in H^1_0(\Om),\quad \int_{\Om}|u|^2dx=\al
\end{cases}
\end{equation}
when $\Om\subset\R^N$ is either a bounded smooth convex domain or all of $\R^N,$ $N\ge 3$, $2<p<2+\frac{4}{N}<q<2^*=\frac{2N}{N-2}$,
the mass $\al>0$ and the parameter $\be\in\R$ are prescribed. The frequency $\la$ is unknown and to be determined. We first investigate  the problem on large bounded convex domains and then consider the compactness of normalized solutions $u_r$ in $B_r$ as the radius $r$ tends to infinity, finally obtaining a normalized solution in $\R^N$.  It would be interesting to investigate whether the bifurcation theory approach developed in \cite{Bartsch-Zhong-Zou2021} can also be employed for \eqref{eq}.

\vskip 0.1in
Before stating our main results, we introduce a few more notations. Let $s_+=\max\{s,0\}$, $s_-=\min\{s,0\}$ for $s\in\R$.
The Aubin-Talenti constant \cite{Aubin} is denoted by $S$, that is, $S$ is the best constant in the Sobolev embedding $D^{1,2}(\R^N)\hookrightarrow L^{2^*}(\R^N)$. Let $C_{N,s}$ be the best constant in the Gagliardo-Nirenberg inequality
$$\|u\|_s^s\le C_{N,s}\|u\|_2^{\frac{2s-N(s-2)}{2}}\|\nabla u\|_2^{\frac{N(s-2)}{2}}\quad \text{for } 2<s<2^*.$$
The functional $E_V$ defined in \eqref{eq:E_V} will be considered on $H^1(\R^N)$ as well as on the subspaces $H^1_0(\Om)$, $\Om\subset\R^N$ open. Our basic assumption on $V$ is:
\begin{itemize}
\item[$(V_0)$] $V\in C(\R^N)\cap L^{\frac{N}{2}}(\R^N)$ is bounded and $\|V_-\|_{{\frac{N}{2}}}< S$.
\end{itemize} 
For some results we require that $V$ is $C^1$ and consider the function
\[
  \ti V:\R^N\to\R,\quad \ti{V}(x)=\nabla V(x)\cdot x.
\]
For $\Om\subset \R^N$ and $r>0$ we set
$$\Om_r=\left\{rx\in \R^N:x\in\Om\right\}$$
and
\[
 S_{r,\al} := S_\al \cap H^1_0(\Om_r) = \big\{u\in H^1_0(\Om_r): \|u\|_{L^2(\Om_r)}^2=\al\big\}.
\]
From now on we assume that $\Om\subset\R^N$ is a bounded smooth convex domain with $0\in\Om$. In our first result we consider the case $\be\le0$.

\begin{thm}\label{boundexist}
Assume $\be\le 0$, $V$ satisfies $(V_0)$, is of class $C^1$ and $\ti V$ is bounded. Then the following hold:
\begin{itemize}
\item[\rm(i)] For every $\al>0$ there exists  $r_{\al}>0$ such that \eqref{eq} on $\Om_r$ with $r>r_\al$ has a {mountain pass type} solution $(\la_{r,\al},u_{r,\al})$ with  $u_{r,\al}>0$ in $\Om_r$ and positive energy $E_V(u_{r,\al})>0$. Moreover, there exists $C_\al>0$  such that
$$\limsup_{r\to\infty}\max_{x\in \Om_r}u_{r,\al}(x)<C_\al.$$
\item[\rm (ii)] If {in addition}  $\|\ti V_+\|_{\frac{N}{2}}< 2S$ then there exists $\ti\al>0$ such that
$$\liminf_{r\to\infty}\la_{r,\al} >0\quad\text{for any \ } 0<\al<\ti \al.$$
\end{itemize}
\end{thm}

\begin{remark}\rm
a) In the proof of Lemma~\ref{moun} we will define $r_\al$ explicitly. From the definition it follows that $r_\al\to\infty$ as $\be\to0$ or as $\al\to\infty$, and that $r_\al\to0$ if $\be\to\infty$ and $\al$ is bounded, or $\al\to0$ and $\be$ is bounded.

b) The solution $(\la_{r,\al},u_{r,\al})$ being of mountain pass type means that $u_{r,\al}$ is a critical point of the constrained functional $E_V|_{S_{r,\al}}$ of mountain pass type with $\la_{r,\al}$ the corresponding Lagrange multiplier. This applies also to other types of critical points obtained in the theorems below.

c) If one is only interested in the existence of solutions for a fixed $r>r_{\al}>0$, then the assumptions on $V$ in Theorem \ref{boundexist} can be weakened to $V\in C^1( \Om_r)$, $V$ and $\ti V$ are bounded in $\Om_r$, and $\|V_-\|_{L^\frac{N}{2}(\Om_r)}< S$. This also applies to Theorems~\ref{boundexi} and \ref{multiplicity} below.
\end{remark}

Next we consider problem \eqref{eq} with $\be>0$, and obtain the existence of solutions of local minimum type and of mountain pass type.

\begin{thm}\label{boundexi}
Assume $\be>0$, $V$ satisfies $(V_0)$, and set
\begin{equation*}
\begin{aligned}
\al_V=&\left(\frac{1-\|V_-\|_{\frac{N}{2}}S^{-1}}{2N(q-p)}\right)^{\frac{N}{2}}\left(\frac{q(4-N(p-2))}{C_{N,q}}\right)^{\frac{4-N(p-2)}{2(q-p)}}
\left(\frac{p(N(q-2)-4)}{\be C_{N,p}}\right)^{\frac{N(q-2)-4}{2(q-p)}}.
\end{aligned}
\end{equation*}
Then the following hold for $0<\al<\al_V$:
\begin{itemize}
\item [\rm(i)] There exists $r_\al>0$ such that \eqref{eq} on $\Om_r$ with $r>r_\al$ has {a local minimum type} solution $(\la_{r,\al},u_{r,\al})$ with  $u_{r,\al}>0$   in $\Om_r$ {and negative energy} $E_V(u_{r,\al})<0$.
 \item [\rm(ii)] There exists $C_\al>0$ such that
$$\limsup_{r\to\infty}\max_{x\in \Om_r}u_{r,\al}(x)<C_\al,\quad\liminf_{r\to\infty}\la_{r,\al}>0.$$
\end{itemize}
\end{thm}

\begin{thm}\label{multiplicity}
Assume $\be>0$, $V$ satisfies $(V_0)$, is of class $C^1$ and $\ti V$ is bounded. Set
\begin{equation*}
\ti \al_V=\left(\frac{2(1-\|V_-\|_{\frac{N}{2}}S^{-1})}{N(q-2)}\right)^{\frac{N}{2}}
\left(\frac{C_{N,q}}{q}A_{p,q}+\frac{C_{N,q}}{q}\right)^{-\frac{N}{2}}\left(\frac{p C_{N,q}}
{\be q C_{N,p} }A_{p,q}\right)^{\frac{N(q-2)-4}{2(q-p)}},
\end{equation*}
where
$$A_{p,q}=\frac{(q-2)(N(q-2)-4)}{(p-2)(4-N(p-2))}.$$
Then the following hold for $0<\al<\ti\al_V$:
 \begin{itemize}
\item [\rm(i)]There exists $\ti r_\al>0$ such that \eqref{eq} in $\Om_r$ admits for $r>r_\al$ a {mountain pass type} solution $(\la_{r,\al},u_{r,\al})$ with $u_{r,\al}>0$ in $\Om_r$ {and positive energy} $E_V(u_{r,\al})>0$. Moreover, there exists $C_\al>0$  such that
$$\limsup_{r\to\infty}\max_{x\in \Om_r}u_{r,\al}(x)<C_\al.$$
\item [\rm(ii)] There exists $0<\bar \al\le \ti a_V$ such that
$$\quad\liminf_{r\to\infty}\la_{r,\al}>0 \quad\text{for any}\ 0<\al\le \bar \al.$$
\end{itemize}
\end{thm}

\begin{remark}\rm
a) $r_\al$ and $\ti r_\al$ in Theorems \ref{boundexi} and \ref{multiplicity} are defined explicitly; see the proof of Theorem~\ref{2.25}, in particular \eqref{2.12}, and the proof of Lemma~\ref{moun1}, in particular \eqref{2.41}.

b) The methods developed in Theorems \ref{boundexist}, \ref{boundexi} and \ref{multiplicity} can also be used to find normalized solutions for other elliptic partial differential equations in large bounded smooth convex domains. We conjecture that the convexity of $\Om$ is not essential for our results.
\end{remark}

Now we obtain the existence and multiplicity of normalized solutions in $\R^N$ by taking $\Om=B_1$, the unit ball centered at the origin in $\R^N$,  and analyzing the compactness of the solutions $u_{r,\al}$ established in Theorems \ref{boundexist}, \ref{boundexi} and \ref{multiplicity} as $r$ tends to infinity. For the passage $r\to\infty$ we require the following condition on $V$.
\begin{itemize}
\item[$(V_1)$] $V$ is of class $C^1$, $\lim\limits_{|x|\to\infty}V(x)=0$, and there exists $\rho\in (0,1)$ such that
\begin{equation}\label{condition}
\liminf_{|x|\to\infty}\inf_{y\in B(x,\rho|x|)}(x\cdot \nabla V(y))e^{\tau |x|}>0\quad\text{for any }\tau>0.
\end{equation}
\end{itemize} 


\begin{thm}\label{exis}
Assume $\be>0$ and $V$ satisfies $(V_0)$-$(V_1)$. Then problem \eqref{eq} with $\Om=\R^N$ admits for any $0<\al< \al_V$, where $\al_V$ is as in Theorem~\ref{boundexi}, a solution $(\la_\al, u_\al)$ with $u_\al>0$, $\la_{\al}>0$, and $E_V(u_\al)<0$.
\end{thm}

{\begin{remark}\rm
a) The assumptions on $V$ in Theorem~\ref{exis} and Theorems~\ref{multipli}, \ref{exis} below hold, for instance, for $V(x)=-\frac{c}{1+|x|^\tau}$ with $\tau>2$ and $c$ small so that $\|V\|_{\frac{N}{2}}<S$. Observe that $(V_1)$ implies $V(x)\le0$ for $|x$ large.  Thus we cover similar potentials as in \cite{Ikoma-Miyamoto2020} but not those in \cite{Bartsch-Molle-Rizzi-Verzini2021} where it was assumed that $V\ge0$.

b) Clearly \eqref{condition} is only concerned with the behavior of $\nabla V$ as $|x|$ tends to infinity. It holds, for instance, if $V(x)\sim-|x|^{-\tau}$ for $|x|$ large with $\tau>2$. The condition $\tau>2$ is needed so that $\|V_-\|_{\frac{N}2}<S$ can be satisfied as required in $(V_0)$. More generally, if $V(x)$ is radial and radially increasing for $|x|$ large, if there are $c_0>0$, $r_1>0$ and $\tau_0>0$ such that
\begin{equation}\label{2-10}
 y\cdot\nabla V(y)\ge c_0|y|^{-\tau_0}\quad \text{for all } |y|\ge r_1,
\end{equation}
and there exist $c_2$, $r_2>0$ such that for any $y$ with $|y|\ge r_2,$
\begin{equation}\label{2-11}
\left|{\xi\cdot\nabla V(y)}\right|\le c_2\left(\frac{y\cdot\nabla V(y)}{|y|}\right)
\quad \text{for all }  \xi\in \R^N\ \text{with}\ \xi\cdot y=0\ \text{and}\ |\xi|=1,
\end{equation}
then \eqref{condition} holds. Indeed, for any $y\in B(x,\rho|x|)$ with $\rho\in(0,1)$ to be determined,
let $\xi=x-\frac{y\cdot x}{|y|^2}y$. Then
$$x=\frac{y\cdot x}{|y|^2}y+\xi\in (span\{y\})\oplus (span\{y\})^\bot.$$
From $y\in B(x,\rho |x|)$ we deduce
$$\frac{1-\rho}{(1+\rho)^2}|y|^2\le y\cdot x\le \frac{1+\rho}{(1-\rho)^2}|y|^2 \quad\text{and}\quad |\xi|\le \frac{\rho}{1-\rho}|y|.$$
In view of \eqref{2-11} we get for $y\in B(x,\rho|x|)$ with $|y|\ge r_2$:
\begin{equation}\label{2-12}
\begin{aligned}
x\cdot\nabla V(y)&=\frac{y\cdot x}{|y|^2}(y\cdot\nabla V(y))+|\xi|\left(\frac{\xi}{|\xi|}\cdot\nabla V(y)\right)\\
&\ge \frac{y\cdot x}{|y|^2}(y\cdot\nabla V(y))-c_2|\xi|\left(\frac{y}{|y|}\cdot\nabla V(y)\right)\\
&\ge\left(\frac{1-\rho}{(1+\rho)^2}-c_2\frac{\rho}{1-\rho}\right)(y\cdot\nabla V(y)).
\end{aligned}
\end{equation}
Note that
$$\frac{1-\rho}{(1+\rho)^2}-c_2\frac{\rho}{1-\rho}\to 1 \quad \text{as\ } \rho \to 0.$$
Hence, we can take $\rho$ small such that
\begin{equation}\label{2-13}
\frac{1-\rho}{(1+\rho)^2}-c_2\frac{\rho}{1-\rho}>\frac{1}{2}.
\end{equation}
Now \eqref{2-10}, \eqref{2-12} and \eqref{2-13} imply for $y\in B(x,\rho|x|)$ with $|y|\ge \max\{r_1, r_2\}$:
$$x\cdot \nabla V(y)\ge \frac{1}{2} (y\cdot\nabla V(y))\ge \frac{c_0}{2}|y|^{-\tau_0}\ge\frac{c_0(1+\rho)^{-\tau_0}}{2}|x|^{-\tau_0}.$$
Therefore \eqref{2-10} and \eqref{2-11} imply \eqref{condition}. Moreover, \eqref{2-11} holds if there exists $r_2>0$ such that $V$ is radial and increasing with respect to $|x|$ in $\R^N\setminus B_{r_2}$.
\end{remark}
}

\begin{thm}\label{multipli}
Assume $\be>0$ and $V$ satisfies $(V_0)$-$(V_1)$. Then \eqref{eq} with $\Om=\R^N$ admits for $0<\al< \bar \al$, $\bar\al>0$ as in Theorem~\ref{multiplicity}(ii), a solution $(\la_\al, u_\al)$ with $u_\al>0$, $\la_{\al}>0$, and $E_V(u_\al)>0$. Moreover, $\lim_{\al\to 0}E_V(u_\al)= \infty$.
\end{thm}

Our last theorem deals with the existence of a solution to \eqref{eq} in $\R^N$ in the case $\be\le 0$.

\begin{thm}\label{exist}
Assume $\be\le 0$, $V$ satisfies $(V_0)$-$(V_1)$, and $\|\ti V_+\|_{\frac{N}{2}}< 2S$. Then problem \eqref{eq} with $\Om=\R^N$ admits for $0<\al< \ti \al$, $\ti\al>0$ as in Theorem~\ref{boundexist}(ii), a solution $(\la_\al, u_\al)$ with $u_\al>0$, $\la_{\al}>0$, and $E_V(u_\al)>0$. Moreover, $\lim_{\al\to 0}E_V(u_\al)= \infty$.
\end{thm}

In the proofs of Theorems \ref{exis}, \ref{multipli} and \ref{exist}, we develop a rather robust method to investigate the existence of normalized solutions in $\R^N$, that can also be used to study the existence of normalized solutions of other elliptic partial differential equations with potentials.

\begin{remark}\rm
a) Theorem~\ref{exist} is new even in the case $\be=0$. In \cite{Bartsch-Molle-Rizzi-Verzini2021} and \cite{Molle-etal:2022} the case $\be=0$ with $V\ge0$ is treated.

b) It is shown in  \cite{Jeanjean-Zhang-Zhong2021} that even for the case
$V\equiv0$ and $\be>0$, there is no nonnegative normalized solution for large $\al$. Moreover, a stronger condition than \eqref{condition} is used to consider the fixed frequency problem in \cite{Sato-Shibata}.
\end{remark}

The rest of this paper is organized as follows. In section 2 we investigate the existence and multiplicity of normalized solutions in large bounded smooth convex domains $\Om_r$ for the cases $\be>0$ and $\be\le 0$, respectively, and  obtain the sign of the Lagrange multiplier. In section 3 we specialize $\Om$ to the unit ball $B_1\subset\R^N$, establish the profile decomposition of the normalized solution $u_{r,\al}$ in $B_r$, and then show the compactness of $u_{r,\al}$ as $r$ tends to infinity.

\section{Normalized solutions in large bounded domains}
In this section, we use different approaches to consider the existence and multiplicity of normalized solutions in large bounded smooth convex domains
since the energy functionals possess different geometric structures according to the sign of $\be$. Consider the problem
\begin{equation}\label{boundeq1}
\begin{cases}
-\De u+V(x)u+\la u=|u|^{q-2}u+\be |u|^{p-2}u\quad \text{in}\ \Om_r,\\
u\in H^1_0(\Om_r),\quad \int_{ \Om_r}|u|^2dx=\al,
\end{cases}
\end{equation}
where $N\ge 3$, $2<p<2+\frac{4}{N}<q<2^*$, the mass $\al>0$ and the parameter $\be\in \R$ are prescribed,
and the frequency $\la$ is unknown and to be determined by the solution. The energy functional $E_r:H_0^1(\Om_r)\to \R$  is defined by
\begin{equation}
E_r(u)=\frac{1}{2}\int_{ \Om_r}|\nabla u|^2dx+\frac{1}{2}\int_{\Om_r}V(x)u^2dx-\frac{1}{q}\int_{ \Om_r}|u|^qdx-\frac{\be}{p}\int_{ \Om_r}|u|^pdx,
\end{equation}
and the mass constraint manifold is defined by
\begin{equation}\label{boundmass}
S_{r,\al}=\left\{u\in H_0^1(\Om_r): \|u\|_2^2=\al\right\}.
\end{equation}

\subsection{Proof of Theorem~\ref{boundexist}}
In this subsection we always assume that the assumptions of Theorem \ref{boundexist} hold. In order to obtain a bounded Palais-Smale sequence, we will use the monotonicity trick inspired by \cite{Jeanjean1999,Borthwick-Chang-Jeanjean-Soave,Chang-Jeanjean-Soave}. For $\frac{1}{2}\le s\le 1$ we define the functional $E_{r,s}: S_{r,\al}\to \R$ by
\begin{equation}\label{2-7}
E_{r,s}(u)=\frac{1}{2}\int_{ \Om_r}|\nabla u|^2dx+\frac{1}{2}\int_{\Om_r}Vu^2dx-\frac{s}{q}\int_{ \Om_r}|u|^qdx-\frac{\be}{p}\int_{ \Om_r}|u|^pdx.
\end{equation}
Note that if $u\in S_{r,\al}$ is a critical point of $E_{r,s}$, then there exists $\la\in \R$
such that $(\la,u)$ is a solution of the equation
\begin{equation}\label{meq1}
\begin{cases}
-\De u+Vu+\la u=s|u|^{q-2}u+\be |u|^{p-2}u\quad \text{in}\ \Om_r,\\
u\in H^1_0(\Om_r),\quad \int_{\Om_r}|u|^2dx=\al.
\end{cases}
\end{equation}

\begin{lemma}\label{moun}
For any $\al>0,$ there exist $r_\al>0$ and  $u^0,u^1\in S_{r_\al,\al}$ such that
\begin{itemize}
\item [\rm (i)]
$E_{r,s}(u^1)\le0$  for any $r>r_\al$ and $s\in \left[\frac{1}{2},1\right]$,
$$\|\nabla u^0\|_2^2 < \left(\frac{2q}{N(q-2)C_{N,q}}\left(1-\|V_-\|_{\frac{N}{2}}S^{-1}\right)
  \al^{\frac{q(N-2)-2N}{4}}\right)^{\frac{4}{N(q-2)-4}} < \|\nabla u^1\|_2^2 $$
and
$$E_{r,s}(u^0)<\frac{(N(q-2)-4)(1-\|V_-\|_{\frac{N}{2}}S^{-1})}{2N(q-2)}
  \left(\frac{2q(1-\|V_-\|_{\frac{N}{2}}S^{-1})}{N(q-2)C_{N,q}\al^\frac{{2q-N(q-2)}}{4}}\right)^{\frac{4}{N(q-2)-4}}.$$

\item[\rm (ii)]If $u\in S_{r,\al}$ satisfies
$$\|\nabla u\|_2^2=\left(\frac{2q}{N(q-2)C_{N,q}}\left(1-\|V_-\|_{\frac{N}{2}}S^{-1}\right)
  \al^{\frac{q(N-2)-2N}{4}}\right)^{\frac{4}{N(q-2)-4}}$$
then there holds
$$E_{r,s}(u)\ge \frac{(N(q-2)-4)(1-\|V_-\|_{\frac{N}{2}}S^{-1})}{2N(q-2)}
  \left(\frac{2q(1-\|V_-\|_{\frac{N}{2}}S^{-1})}{N(q-2)C_{N,q}\al^{\frac{2q-N(q-2)}{4}}}\right)^{\frac{4}{N(q-2)-4}}.$$

\item [\rm(iii)] Set
$${m_{r,s}(\al)}=\inf_{\gamma\in \Gamma_{r,\al}}\sup_{t\in [0,1]}E_{r,s}(\gamma(t)),$$
{with
$$\Gamma_{r,\al}=\left\{\gamma\in C([0,1],S_{r,\al}):\gamma(0)=u^0, \gamma(1)=u^1\right\}.$$}
Then
$$\frac{(N(q-2)-4)(1-\|V_-\|_{\frac{N}{2}}S^{-1})}{2N(q-2)}
\left(\frac{2q(1-\|V_-\|_{\frac{N}{2}}S^{-1})}{N(q-2)C_{N,q}\al^{\frac{2q-N(q-2)}{4}}}\right)^{\frac{4}{N(q-2)-4}}\le m_{r,s}(\al)\le T_\al,$$
where ${T_\al}=\mathop{\max}\limits_{t\in\R^+}h(t)$, the function $h:\R^+\to\R$ being defined by
$$h(t)=\frac{1}{2}\left(1+\|V\|_{\frac{N}{2}}S^{-1}\right)t^2\theta\al-\frac{\be C_{N,p}}{p}\al^{\frac{p}{2}}\theta^{\frac{N(p-2)}{4}}
t^{\frac{N(p-2)}{2}}-\frac{1}{2q}\al^{\frac{q}{2}}|\Om|^{\frac{2-q}{2}}t^{\frac{N(q-2)}{2}}.$$
Here $\theta$ is the principal eigenvalue of $-\De$ with Dirichlet boundary conditions in $\Om$, and $|\Om|$ is the volume of $\Om$.
\end{itemize}
\end{lemma}

\vskip 0.1in
\begin{proof}(i)
Clearly the set $S_{r,\al}$ is path connected. Let $v_1\in S_{1,\al}$ be the positive eigenfunction associated to $\theta$. Then we have
\begin{equation}
\int_{\Om}|\nabla v_1|^2dx=\theta\al \quad\text{and}\quad \int_{\Om}| v_1|^qdx\ge\al^{\frac{q}{2}}|\Om|^{\frac{2-q}{2}}.
\end{equation}
by the H\"older inequality. Setting $v_t(x)=t^{\frac{N}{2}}v_1(tx)$ for $x\in \Om_{\frac{1}{t}}$ and $t>0$  there holds
\begin{equation}\label{2.15}
\begin{aligned}
E_{\frac{1}{t},s}(v_t) &\leq \frac{1}{2}\left(1+\|V\|_{\frac{N}{2}}S^{-1}\right)t^2\int_{\Om}|\nabla v_1|^2dx
-\frac{\be C_{N,p}}{p}\al^{\frac{2p-N(p-2)}{4}}\left(t^2\int_{\Om}|\nabla v_1|^2dx\right)^{\frac{N(p-2)}{4}}\\
&\hspace{1cm} -\frac{1}{2q}t^{\frac{N(q-2)}{2}}\int_{\Om}| v_1|^qdx\\
&\le h(t).\\
\end{aligned}
\end{equation}
Note that since $2<p<2+\frac{4}{N}<q<2^*$ and $\be<0$ there exist $0<_\al<t_0$ such that $h(t_0)=0$, $h(t)<0$ for any $t>t_0$,
$h(t)>0$ for any $0<t<t_0$ and $h(T_\al)=\mathop{\max}\limits_{t\in\R^+}h(t)$.
As a consequence, there holds
\begin{equation}\label{2.30}
E_{r,s}(v_{t_0})= E_{\frac{1}{t_0},s}(v_{t_0})\le h(t_0)=0
\end{equation}
for any {$r\ge\frac{1}{t_0}$} and $s\in\left[\frac{1}{2},1\right]$. Moreover, there exists $0<t_1<T_\al$ such that
\begin{equation}\label{2.31}
h(t)<\frac{(N(q-2)-4)(1-\|V_-\|_{\frac{N}{2}}S^{-1})}{2N(q-2)}
\left(\frac{2q(1-\|V_-\|_{\frac{N}{2}}S^{-1})}{N(q-2)C_{N,q}\al^\frac{{2q-N(q-2)}}{4}}\right)^{\frac{4}{N(q-2)-4}} \text{ for } t\in [0,t_1].
\end{equation}
On the other hand, it follows from the  Gagliardo-Nirenberg inequality and the H\"older inequality that
\begin{equation}\label{2.23}
E_{r,s}(u)\ge \frac{1}{2}\left(1-\|V_-\|_{\frac{N}{2}}S^{-1}\right)\int_{\Om_r}|\nabla u|^2dx-
\frac{C_{N,q}\al^{\frac{2q-N(q-2)}{4}}}{q}\left(\int_{\Om_r}|\nabla u|^2dx\right)^{\frac{N(q-2)}{4}}.
\end{equation}
Let
$$f(t)=\frac{1}{2}\left(1-\|V_-\|_{\frac{N}{2}}S^{-1}\right)t-
\frac{C_{N,q}\al^{\frac{2q-N(q-2)}{4}}}{q}t^{\frac{N(q-2)}{4}},$$
and 
$$\ti t =\left(\frac{2q}{N(q-2)C_{N,q}}\left(1-\|V_-\|_{\frac{N}{2}}S^{-1}\right)\al^{\frac{q(N-2)-2N}{4}}\right)^{\frac{4}{N(q-2)-4}}.$$
Then $f$ is increasing in $(0,\ti t)$ and decreasing in $(\ti t,\infty)$, and
$$f(\ti t)= \frac{(N(q-2)-4)(1-\|V_-\|_{\frac{N}{2}}S^{-1})}{2N(q-2)}
\left(\frac{2q(1-\|V_-\|_{\frac{N}{2}}S^{-1})}{N(q-2)C_{N,q}\al^{\frac{2q-N(q-2)}{4}}}\right)^{\frac{4}{N(q-2)-4}}.$$
For $r\ge\ti r_\al:=\max\left\{\frac{1}{t_1},\sqrt{\frac{2\theta\al}{\ti t}}\right\}$ we have $v_{\frac{1}{\ti r_\al}}\in S_{r,\al}$ and
\begin{equation}\label{2.28}
\|\nabla v_{\frac{1}{\ti r_\al}}\|_2^2=\left(\frac{1}{\ti r_\al}\right)^2\|\nabla v_1\|_2^2<\left(\frac{2q}{N(q-2)C_{N,q}}\left(1-\|V_-\|_{\frac{N}{2}}S^{-1}\right)\al^{\frac{q(N-2)-2N}{4}}\right)^{\frac{4}{N(q-2)-4}}.
\end{equation}
Moreover, there holds
\begin{equation}\label{2.29}
E_{\ti r_\al,s}(v_{\frac{1}{\ti r_\al}})\le h\left(\frac{1}{\ti r_\al}\right)\le h(t_1).
\end{equation}
Setting $u^0=v_{\frac{1}{\ti r_\al}},\ u^1=v_{t_0}$ and
\begin{equation}\label{2-16}
r_\al=\max\left\{\frac{1}{t_0}, \ti r_\al\right\}
\end{equation}
the statement (i) holds due to \eqref{2.28},\eqref{2.29}, \eqref{2.31} and \eqref{2.30}.

\vskip 0.1in
(ii) Statement (ii) holds by \eqref{2.23} and a  direct calculation.
\vskip 0.1in

 (iii) Since  $E_{r,s}(u^1)\le0$ for any $\gamma\in\Gamma_{r,\al}$, we have
 $$\|\nabla \gamma(0)\|_2^2<\ti t< \|\nabla \gamma(1)\|_2^2.$$
 It then follows from \eqref{2.23} that
 $$
 \begin{aligned}
  \max_{t\in[0,1]}E_{r,s}(\gamma(t))
     &\ge f(\ti t)\\
     &= \frac{(N(q-2)-4)(1-\|V_-\|_{\frac{N}{2}}S^{-1})}{2N(q-2)}
             \left(\frac{2q(1-\|V_-\|_{\frac{N}{2}}S^{-1})}{N(q-2)C_{N,q}\al^\frac{{2q-N(q-2)}}{4}}\right)^{\frac{4}{N(q-2)-4}}
 \end{aligned}
$$
for any $\gamma\in\Gamma_{r,\al}$, hence the first inequality in (iii) holds. Now we define a path $\gamma\in \Gamma_{r,\al}$ by
 \begin{equation*}
\gamma(\tau)(x)=\left(\tau t_0+(1-\tau)\frac{1}{\ti r_\al}\right)^{\frac{N}{2}}v_1\left(\left(\tau t_0+(1-\tau)\frac{1}{\ti r_\al}\right) x\right)
\end{equation*}
for $\tau\in [0,1]$ and $x\in \Om_r$. Then by \eqref{2.15} we have $m_{r,s}(\al)\le T_\al,$  where  $T_\al=\mathop{\max}\limits_{t\in\R^+}h(t)$.
 Note that $T_\al$ is independent of $r$ and $s$. This finishes the proof of (iii). 
\end{proof}

In view of Lemma \ref{moun}, the energy functional $E_{r,s}$ possesses the mountain pass geometry. Next we recall a theorem from \cite{Borthwick-Chang-Jeanjean-Soave,Chang-Jeanjean-Soave}.

\begin{thm}[Theorem 1, \cite{Borthwick-Chang-Jeanjean-Soave}]\label{criticalprinciple}
Let $(E,\langle\cdot,\cdot\rangle)$ and $(H,(\cdot,\cdot))$ be two infinite-dimensional Hilbert spaces and assume there are continuous injections
$$E\hookrightarrow H\hookrightarrow E^\prime.$$
Let
$$\|u\|^2=\langle u,u\rangle,\quad |u|^2=(u,u)\quad \text{for }u\in E,$$
and
$$S_\mu=\{u\in E:|u|^2=\mu\},\quad T_uS_\mu=\{v\in E:(u,v)=0\}\quad \text{for } \mu\in(0,+\infty).$$
Let
$I\subset(0,+\infty)$ be an interval and consider a family of $C^2$ functionals $\Phi_\rho: E\to\R$
of the form
$$\Phi_\rho(u)=A(u)-\rho B(u),\quad \text{for } \rho\in I,$$
with $B(u)\ge 0$ for every $u\in E$, and
\begin{equation}\label{2-8}
  A(u)\to+\infty\quad \text{or}\quad B(u)\to+\infty\ \ \text{as}\ u\in E\ \text{and}\ \|u\|\to+\infty.
\end{equation}
Suppose moreover that  $\Phi^\prime_\rho$ and $\Phi^{\prime\prime}_\rho$ are
$\tau-$H\"older continuous, $\tau\in(0, 1]$, on bounded sets in the following sense: for every $R>0$ there exists $M=M(R)>0$
such that
\begin{equation}\label{2-14}
\|\Phi^\prime_\rho(u)-\Phi^\prime_\rho(v) \|\le M \|u-v\|^\tau \quad\text{and}\quad
  \|\Phi^{\prime\prime}_\rho(u)-\Phi^{\prime\prime}_\rho(v) \|\le M \|u-v\|^\tau
\end{equation}
for every $u,v\in B(0,R)$. Finally, suppose that there exist $w_1,w_2\in S_\mu$ independent of $\rho$ such that
$$c_\rho:=\inf_{\gamma\in\Gamma}\max_{t\in [0,1]}\Phi_\rho(\gamma(t))>\max\{\Phi_\rho(w_1),\Phi_\rho(w_2)\}\quad \text{for all }\rho\in I$$
where
$$\Gamma=\{\gamma\in C([0,1],S_\mu):\gamma(0)=w_1,\ \gamma(1)=w_2\}.$$
Then for almost every $\rho\in I$ there exists a sequence $\{u_n\}\subset S_\mu$ such that
\begin{itemize}
\item[\rm (i)] $\Phi_\rho(u_n)\to c_\rho,$
\item[\rm (ii)] $\Phi^\prime_\rho|_{S_\mu}(u_n)\to0,$
\item[\rm (iii)] $\{u_n\}$ is bounded in $E$.
\end{itemize}
\end{thm}

\begin{remark}\rm
One can see \cite{Borthwick-Chang-Jeanjean-Soave,Chang-Jeanjean-Soave} for more details on assumption \eqref{2-14}.
\end{remark}

Lemma \ref{moun} and Theorem \ref{criticalprinciple} yield the following result.

\begin{thm}\label{existmeq}
For any $\al>0$, let $r>r_\al$, where $r_\al$ is defined in Lemma \ref{moun}.
Then  problem \eqref{meq1} admits a solution $(\la_{r,s},u_{r,s})$ for almost every $s\in\left[\frac{1}{2},1\right].$
Moreover, $u_{r,s}\ge 0$ and $E_{r,s}(u_{r,s})=m_{r,s}(\al)$.
\end{thm}

\begin{proof}
We apply Theorem \ref{criticalprinciple} to $E_{r,s},$ with $\Gamma_{r,\al}$ defined in Lemma \ref{moun},
$$A(u)=\frac{1}{2}\int_{\Om_r}|\nabla u|^2dx+\frac{1}{2}\int_{\Om_r}V(x) u^2dx-\frac{\be}{p}\int_{\Om_r}|u|^pdx\quad\text{and}\quad
   B(u)=\frac{1}{q}\int_{\Om_r}|u|^qdx.$$
Note that the assumptions in Theorem \ref{criticalprinciple} hold due to $\be\le0$ and Lemma \ref{moun}.
Hence, for almost every $s \in [\frac{1}{2},1],$
there exists a bounded  Palais-Smale sequence $\{u_n\}$:
$$E_{r,s}(u_n)\to m_{r,s}(\al) \quad\text{and}\quad E_{r,s}^\prime(u_n)|_{T_{u_n}S_{r,\al}}\to 0$$
where $T_{u_n}S_{r,\al}$ denotes the tangent space of $S_{r,\al}$ at $u_n$. Then
$$\la_n=-\frac{1}{\al}\left(\int_{\Om_r} |\nabla u_n|^2dx+\int_{\Om_r} Vu_n^2dx-\be \int_{\Om_r} | u_n|^pdx-s\int_{\Om_r} | u_n|^qdx\right),$$
is bounded and
\begin{equation}\label{2.8}
E_{r,s}^\prime(u_n)+\la_nu_n\to 0 \quad\text{in\;} H^{-1}(\Om_r).
\end{equation}
Moreover, there exist $u_0\in H_0^1(\Om_r)$ and  $\la\in\R$
such that up to a subsequence,
$$\la_n\to \la,\quad u_n\rightharpoonup u_0\quad \text{in } H_0^1(\Om_r)\quad\text{and}\quad
   u_n\to u_0\quad \text{in } L^t(\Om_r) \text{ for all }2\le t< 2^*,$$
and $u_0$ satisfies
\begin{equation*}
\begin{cases}
-\De u_0+Vu_0+\la u_0=s|u_0|^{q-2}u_0+\be |u_0|^{p-2}u_0\quad\text{in\ }\Om_r,\\
u_0\in H^1_0(\Om_r),\quad \int_{\Om_r}|u_0|^2dx=\al.
\end{cases}
\end{equation*}
In view of \eqref{2.8}, we have or $n\to\infty$:
$$E_{r,s}^\prime(u_n)u_0+\la_n\int_{\Om_r} u_n u_0dx\to 0 \quad\text{and}\quad
  E_{r,s}^\prime(u_n)u_u+\la_n \al\to 0.$$
Note that
$$\lim_{n\to\infty}\int_{\Om_r}V(x)u_n^2dx= \int_{\Om_r}V(x)u_0^2dx.$$
As a result we get $u_n\to u_0$ \text{in} $H_0^1(\Om_r)$, hence $E_{r,s}(u_{0})=m_{r,s}(\al)$. In order to obtain a nonnegative normalized solution
we only need to modify the proof of Theorem \ref{criticalprinciple} established in \cite{Borthwick-Chang-Jeanjean-Soave,Chang-Jeanjean-Soave}. In fact,
for almost every $s\in[\frac{1}{2},1]$, the derivative $m^\prime_{r,s}$ with respect to $s$ is well defined since the
function $s\mapsto m_{r,s}$ is nonincreasing, where $m_{r,s}$  denotes $m_{r,s}(\al)$  for fixed $\al$.
Let $s$ be such that $m^\prime_{r,s}$ exists and $\{s_n\}\subset[\frac{1}{2},1]$
be a monotone increasing sequence converging to s.
Similar to the proof of  Theorem \ref{criticalprinciple}  established in  \cite{Borthwick-Chang-Jeanjean-Soave,Chang-Jeanjean-Soave}, there
exist $\{\gamma_n\}\subset\Gamma_{r,\al}$ and $K=K(m^\prime_{r,s})$ such that:
\begin{itemize}
\item[(i)] if $E_{r,s}(\gamma_n(t))\ge m_{r,s}-(2-m_{r,s}^\prime)(s-s_n),$
 then $\int_{\Om_r}|\nabla\gamma_n(t)|^2dx\le K$.
 \item[(ii)] $\mathop{\max}\limits_{t\in[0,1]} E_{r,s}(\gamma_n(t))\le m_{r,s}-(2-m_{r,s}^\prime)(s-s_n).$
\end{itemize}
Setting $\ti{\gamma}_n(t)=|\gamma_n(t)|$ for any $t\in [0,1]$ we have $\{\ti{\gamma}_n\}\subset\Gamma_{r,\al}$.
Observe that $\|\nabla |u\||_2^2\le\|\nabla u\|_2^2$ for any $u\in H^1(\R^N)$. Now we have:
\begin{itemize}
\item [(I)]if $E_{r,s}(\ti\gamma_n(t))\ge m_{r,s}-(2-m_{r,s}^\prime)(s-s_n),$ then $E_{r,s}(\gamma_n(t))\ge m_{r,s}-(2-m_{r,s}^\prime)(s-s_n).$
By (i), there holds $\int_{\Om_r}|\nabla\gamma_n(t)|^2dx\le K$, and hence $\int_{\Om_r}|\nabla\ti\gamma_n(t)|^2dx\le K.$
Thus (i) also holds for $\ti\gamma_n$.
 \item[(II)] $\mathop{\max}\limits_{t\in[0,1]} E_{r,s}(\ti\gamma_n(t))\le\mathop{\max}\limits_{t\in[0,1]} E_{r,s}(\gamma_n(t))\le m_{r,s}-(2-m_{r,s}^\prime)(s-s_n).$
\end{itemize}
By replacing $\gamma_n$ with $\ti \gamma_n$ in the proof of Theorem \ref{criticalprinciple} established in \cite{Borthwick-Chang-Jeanjean-Soave,Chang-Jeanjean-Soave}, we obtain a nonnegative bounded Palais-Smale sequence $\{u_n\}$. Consequently, there exists a nonnegative normalized solution to \eqref{meq1} for almost every $s\in\left[\frac{1}{2},1\right]$ as above.  The proof is complete.
\end{proof}

In order to obtain a solution of \eqref{boundeq1} we first prove a uniform estimate for the solutions of \eqref{meq1} established in Theorem \ref{existmeq}.

\begin{lemma}\label{priori}
If $(\la, u)\in \R\times S_{r,\al}$ is a solution of \eqref{meq1} established in Theorem \ref{existmeq} for some $r$ and $s$,
then
$$\int_{\Om_r}|\nabla u|^2dx\le \frac{4N}{N(q-2)-4}\left(\frac{q-2}{2}T_\al+\al\left(\frac{1}{2N}\|\ti V\|_\infty+\frac{q-2}{4}\|V\|_\infty\right)\right),$$
where the constant $T_\al$ is defined in (iii) of Lemma \ref{moun} and is independent of $r$ and $s.$
\end{lemma}

\begin{proof}
Since $u$ is a solution of \eqref{meq1}, there holds
\begin{equation}\label{2.4}
\begin{aligned}
\int_{\Om_r}|\nabla u|^2dx+\int_{ \Om_r}Vu^2dx=s\int_{\Om_r}|u|^qdx+\be\int_{ \Om_r}|u|^pdx-\la\int_{\Om_r}|u|^2dx.
\end{aligned}
\end{equation}
The Pohozaev identity implies 
\begin{equation*}
\begin{aligned}
&\frac{N-2}{2N}\int_{\Om_r}|\nabla u|^2dx+\frac{1}{2N}\int_{\partial \Om_r}|\nabla u|^2(x\cdot \mathbf{n})d\sigma+
\frac{1}{2N}\int_{\Om_r}(\ti V)u^2+\frac{1}{2}\int_{\Om_r}Vu^2dx\\
&\hspace{2cm}= -\frac{\la}{2}\int_{ \Om_r}|u|^2dx+\frac{s}{q}\int_{ \Om_r}|u|^qdx+\frac{\be}{p}\int_{ \Om_r}|u|^pdx,
\end{aligned}
\end{equation*}
where $\mathbf n$ denotes the outward unit normal vector on $\partial \Om_r.$ It then follows from $\be\le0$ that
\begin{align*}
&\frac{1}{N}\int_{\Om_r}|\nabla u|^2dx-\frac{1}{2N}\int_{\partial \Om_r}|\nabla u|^2(x\cdot \mathbf{n})d\sigma-
\frac{1}{2N}\int_{\Om_r}(\nabla V\cdot x)u^2dx\\
&\hspace{2cm}=\frac{(q-2)s}{2q}\int_{ \Om_r}|u|^qdx+\frac{\be(p-2)}{2p}\int_{ \Om_r}|u|^pdx\\
&\hspace{2cm}\ge\frac{q-2}{2}\left(\frac{s}{q}\int_{ \Om_r}|u|^qdx+\frac{\be}{p}\int_{ \Om_r}|u|^pdx\right)\\
&\hspace{2cm}=\frac{q-2}{2}\left(\frac{1}{2}\int_{\Om_r}|\nabla u|^2dx+\frac{1}{2}\int_{ \Om_r}Vu^2dx-m_{r,s}(\al)\right).
\end{align*}
Consequently, we have
\begin{align*}
\frac{q-2}{2}m_{r,s}(\al)
&\ge\frac{q-2}{2}\left(\frac{1}{2}\int_{\Om_r}|\nabla u|^2dx+\frac{1}{2}\int_{ \Om_r}Vu^2dx\right)\\
&\hspace{.5cm} -\left(\frac{1}{N}\int_{\Om_r}|\nabla u|^2dx-\frac{1}{2N}\int_{\partial \Om_r}|\nabla u|^2(x\cdot \mathbf{n})d\sigma-
\frac{1}{2N}\int_{\Om_r}(\nabla V\cdot x)u^2dx\right)\\
&\ge\frac{N(q-2)-4}{4N}\int_{\Om_r}|\nabla u|^2dx-\al\left(\frac{1}{2N}\|\nabla V\cdot x\|_\infty+\frac{q-2}{4}\|V\|_\infty\right),
\end{align*}
where the last inequality holds since $x\cdot \mathbf{n}(x)\ge 0$ for any $x\in\partial \Om_r$ due to the convexity of $\Om_r.$
By a direct calculation and Lemma \ref{moun} we get the inequality.
\end{proof}

Now,  we obtain a solution of \eqref{boundeq1} by letting $s\to 1$.

\begin{thm}\label{existbound}
For every $\al>0$ problem \eqref{boundeq1} has a solution $(\la_{r},u_{r})$ provided $r>r_\al$ where $r_\al$ is as in Lemma \ref{moun}. 
Moreover, $u_{r}\ge0$ in $\Om_r$.
\end{thm} 

\begin{proof} It follows from Theorem \ref{existmeq} that there is a nonnegative solution $(\la_{r,s},u_{r,s})$ to \eqref{meq1}
for almost every $s\in \left[\frac{1}{2},1\right]$. In view of Lemma \ref{priori}, $\{u_{r,s}\}$ is bounded. By an argument similar to that in Theorem~\ref{existmeq}, there exist $u_r\in S_{r,\al}$ and $\la_r$ such that up to a subsequence,
 $$\la_{r,s}\to\la_r \quad\text{and}\quad u_{r,s}\to u_r \quad\text{in}\ H_0^1(\Om_r)\quad \text{as}\ s\to 1.$$
Hence $u_r$ is a nonnegative solution of problem \eqref{boundeq1}.
\end{proof}

Before analyzing the Lagrange multiplier, we first establish an a priori estimate for the solutions of \eqref{boundeq1}.

\begin{lemma}\label{blow}
If $\{(\la_r, u_r)\}$ is a family of nonnegative solutions of \eqref{boundeq1} such that $\|u_r\|_{H^1}\le C$ with $C>0$ independent of $r$, 
then $\limsup\limits_{r\to\infty}\|u_r\|_\infty<\infty$  is bounded.
\end{lemma}

\begin{proof}
It follows from the regularity theory of elliptic partial differential equations that $u_r\in C(\Om_r)$.
Assume to the contrary that there exist a sequence, for simplicity denoted by $\{u_r\}$, and $x_r\in \Om_r$ such that
$$M_r:=\max_{x\in \Om_r}u_r(x)=u_r(x_r)\to \infty\quad \text{as}\ r\to \infty.$$
Suppose without loss of generality that, up to a subsequence, $\mathop{\lim}\limits_{r\to\infty}\frac{x_r}{|x_r|}=(1,0,\dots,0)$. Set
$$v_r(x)=\frac{u_r(x_r+\tau_rx)}{M_r}\quad \text{for } x\in\Sigma^r:=\{x\in\R^N:x_r+\tau_rx\in \Om_r\},$$
where $\tau_r=M_r^{\frac{2-q}{2}}.$ Then $\tau_r\to0$, $\|v_r\|_{L^\infty(\Sigma^r)}\le1$, and $v_r$ satisfies
\begin{equation}\label{2.27}
-\De v_r+\tau_r^2V(x_r+\tau_rx)v_r+\tau_r^2\la_rv_r=|v_r|^{q-2}v_r+\be\tau_r^{\frac{2(q-p)}{q-2}}|v_r|^{p-2}v_r\quad \text{in\ }\Sigma^r.
\end{equation}
In view of \eqref{boundeq1}, the   Gagliardo-Nirenberg inequality and $\|u_r\|_{H^1}\le C$ with $C$ independent of $r$, we infer that the sequence $\{\la_r\}$ is bounded. It then follows from the regularity theory of elliptic partial differential equations and the Arzela-Ascoli theorem that there exists $v$ such that up to a subsequence
$$v_r\rightharpoonup v\quad\text{in}\; H^1_0(\Sigma) \quad\text{and}\quad
  v_r\to v\quad\text{in}\; C^\be_{loc}(\Sigma)\ \text{for some }\be\in (0,1), $$
where $\Sigma:=\mathop{\lim}\limits_{r\to\infty} \Sigma^r$.

We  claim that
\begin{equation}\label{2-17}
\liminf_{r\to\infty}\frac{{\rm dist}(x_r,\partial \Om_r)}{\tau_r}=\liminf_{r\to\infty}\frac{|y_r-x_r|}{\tau_r}\ge d>0,
\end{equation}
where $y_r\in \partial \Om_r$ is such that ${\rm dist}(x_r,\partial \Om_r)=|y_r-x_r|$ for any large $r$. If this is not true, let
$$p_r=\frac{y_r-x_r}{\tau_r}\in\partial\Sigma^r,$$
so that $p_r\to 0$, $v_r(p_r)=0$ and $v_r(0)=1.$ If  $v_r$ is uniformly (in $r$) locally Lipschitz continuous at $p_r$ we obtain a contradiction because
$$|v_r(p_r)-v_r(0)|\le L|p_r-0|\to 0$$
contradicts $v_r(p_r)=0$ and $v_r(0)=1$. In order to obtain the uniformly locally Lipschitz continuity of $v_r$ at $p_r$ we first translate $p_r$ to the
point $(-1,0,\cdots,0)$ such that the unit ball $B_1$ is exterior tangent to $\partial \Sigma^r$ at $(-1,0,\cdots,0)$.
Let $\phi_1(x)=\frac{1}{2N}(1-|x|^2)$ for any $x\in B_1$, then $$-\De \phi_1=1\quad \text{in}\; B_1.$$
Let $\phi_2$ be the Kelvin transformation of $\phi_1,$ that is,
$$\phi_2(x)=\frac{1}{|x|^{N-2}}\phi_1\left(\frac{x}{|x|^2}\right)\quad \text{for } x\in B_1^c,$$
so that
$$-\De \phi_2=\frac{1}{|x|^{N+2}}\quad \text{in}\; B_1^c.$$
Let $\phi\in C_c^\infty(\R^N)$ satisfy $0\le\phi\le 1$ in $\R^N$,
$\phi=0$ in $\overline {B_1}$, $\phi(x)>0$ in $ B_3\setminus\overline{ B_1}$, and $\phi(x)=1$ for $x\in \partial B_3$.
Then there exists $C>0$ such that
$$-\De \phi \ge -C \quad \text{in}\; B_3\setminus\overline{ B_1}.$$
Define $h(x)=t\phi_2(x)+\phi(x)$, where  $t>0$ is to be determined. Note that $h$ satisfies
$$-\De h=\frac{t}{|x|^{N+2}}-C\ge \frac{t}{3^{N+2}}-C\quad \text{in}\; B_3\setminus \overline{ B_1}, $$
$$h(x)=\frac{t}{2N|x|^{N-2}}\left(1-\frac{1}{|x|^2}\right)+\phi(x)\quad\text{for } x\in \partial \Sigma^r,$$
and
$$h(x)=\frac{4t}{N3^N}+1\quad\text{for } x\in \partial B_3.$$
Hence we can choose $t$ large  such that for any  $r$ large enough:
\begin{equation}
\begin{cases}
-\De (h-v_r)\ge 0\quad &\text{in}\ \Sigma^r\cap B_3,\\
h-v_r\ge 0&\text{on}\ \partial(\Sigma^r\cap B_3).
\end{cases}
\end{equation}
The maximum principle then implies that
$$h\ge v_r\quad \text{in}\ \Sigma^r\cap B_3.$$
This yields
$$0\le v_r(x)-v_r(p_r)=v_r(x)\le h(x)=h(x)-h(p_r)\quad \text{in}\ \Sigma^r\cap B_3.$$
Consequently,  we only need to prove the  locally Lipschitz continuity of $\phi_2$ at $p_r$. Note that
\begin{equation}
\begin{aligned}
\phi_2(x)-\phi_2(p_r)&=\frac{1}{2N|x|^{N-2}}\left(1-\frac{1}{|x|^2}\right)
 =\frac{1}{2N|x|^N}(|x|+1)(|x|-1)\\
&\le L|x-p_r|
\end{aligned}
\end{equation}
for $x\in B(p_r,\frac{1}{4})\cap B_1^c$. Here the last inequality holds since we have translated $p_r$ to $(-1,0,\cdots,0).$
Therefore the claim \eqref{2-17} holds. As a result, by letting $r\to\infty$ in \eqref{2.27}, we obtain that  $v\in H_0^1(\Sigma)$ is a nonnegative solution of
\begin{equation}
-\De v=|v|^{q-2}v\quad \text{in } \Sigma
\end{equation}
where 
\[
  \Si = \begin{cases} 
             \R^N &\text{ if $\liminf_{r\to\infty}\frac{{\rm dist}(x_r,\partial \Om_r)}{\tau_r}=\infty$}\\
             \{x\in \R^N: x_1>-d\} &\text{ if $\liminf_{r\to\infty}\frac{{\rm dist}(x_r,\partial \Om_r)}{\tau_r}>0$.}
           \end{cases}
\]
It then follows from the Liouville theorems (see e.g.\ \cite{Esteban-Lions}) that $v=0$ in $H$, which contradicts
$v(0)=\mathop{\lim}\limits_{r\to\infty}v_r(0)=1$.
\end{proof}

\begin{remark}\label{2.26}\rm
Note that the proof of Lemma \ref{blow} does not depend on $\be.$
\end{remark}

\begin{lemma}\label{multi} Let $(\la_{r,\al},u_{r,\al})$ be the solution of \eqref{boundeq1} from Theorem~\ref{existbound}. If $\|\ti V_+\|_{\frac{N}{2}}\le 2S$  then
there exists  $\bar\al>0$  such that
$$\mathop{\liminf}\limits_{r\to\infty}\la_{r,\al}>0\quad
\text{for } 0<\al<\bar\al.$$
\end{lemma}

\begin{proof}
Let $(\la_{r,\al},u_{r,\al})$ be the solution of \eqref{boundeq1} established in Theorem~\ref{existbound}. The regularity theory of elliptic partial differential equations yields $u_{r,\al}\in C(\Om_r)$. In view of Lemma \ref{blow}, there holds
\begin{equation*}
\mathop{\limsup}\limits_{r\to\infty} \max_{\Om_r} u_{r,\al}<\infty.
\end{equation*}
Setting $$g(\al)=\liminf_{r\to\infty} \max_{\Om_r} u_{r,\al}$$
we claim that there is $\al_1>0$ such that $g(\al)>0$ for any $0<\al<\al_1$. Assume to the contrary that there exists a sequence $\{\al_k\}$ tending to $0$ as $k\to\infty$ such that $g(\al_k)=0$ for any $k$, that is,
\begin{equation}\label{2-4}
\liminf_{r\to\infty}\max_{\Om_r} u_{r,\al_k} = 0\quad \text{for\ any }k.
\end{equation}
As a consequence of (iii) in Lemma \ref{moun},
\begin{equation}\label{2-15}
E_{r}(u_{r,\al_k})=m_{r,1}(\al_k)\to \infty \quad \text{as\ } k\to0 \quad\text{for\ any\ }r>r_{\al_k}.
\end{equation}
For any given $k$, it follows from \eqref{2-4} and $u_{r,\al_k}\in S_{r,\al_k}$ that up to a subsequence,
\begin{equation}\label{2.19-1}
\int_{\Om_r}|u_{r,\al_k}|^sdx\le |\max_{\Om_{r}}u_{r,\al_k}|^{s-2}\al_k\to 0\quad\text{as}\ r\to\infty
\end{equation}
for any $s>2$. Hence, for any given large $k$, there exists $ \bar r_k> r_{\al_k}$ such that
$$\left|\frac{1}{q}\int_{ \Om_r}|u_{r,\al_k}|^qdx+\frac{\be}{p}\int_{\Om_r}|u_{r,\al_k}|^pdx\right|<\frac{m_{r,1}(\al_k)}{2}\quad \text{for\ any\ } r\ge\bar r_k.$$
In view of \eqref{2-15} and $E_{r}(u_{r,\al_k})=m_{r,1}(\al_k)$, we further have
\begin{equation}\label{2-5}
\int_{\Om_r}|\nabla u_{r,\al_k}|^2dx+\int_{\Om_r}V u_{r,\al_k}^2dx\ge \frac{m_{r,1}(\al_k)}{2}
\quad \text{for\ any\ large\ } k \text{\ and\ }r\ge \bar r_k.
\end{equation}
It follows from \eqref{2-4}, \eqref{2.19-1} and \eqref{2-5} that there exists $r_k\ge\bar r_k$ with $r_k\to \infty$ as $k\to \infty $ such that
\begin{equation}\label{2-1}
0=\lim_{k\to\infty}\max_{\Om_{r_k}}u_{r_k,\al_k} = 0,
\end{equation}
\begin{equation}\label{2.19-2}
\int_{\Om_{r_k}}|u_{r_k,\al_k}|^sdx\le |\max_{ \Om_{r_k}}u_{r_k,\al_k}|^{s-2}\al_k\to 0\quad\text{as}\ k\to\infty \text{\ for\ any\ } s>2,
\end{equation}
and
\begin{equation}\label{2-2}
\int_{\Om_{r_k}}|\nabla u_{r_k,\al_k}|^2dx+\int_{\Om_r}V u_{r_k,\al_k}^2dx\to\infty \quad\text{as}\ k\to\infty.
\end{equation}
By \eqref{boundeq1}, \eqref{2.19-2} and \eqref{2-2}, we have
\begin{equation}\label{2.21}
\la_{r_k,\al_k}\to-\infty \quad \text{as}\ k \to \infty.
\end{equation}
Now \eqref{boundeq1} implies
\begin{equation*}
-\De u_{r_k,\al_k}+\left(\|V\|_{\infty}+\frac{\la_{r_k,\al_k}}{2}\right)u_{r_k,\al_k}
\ge\left(-\frac{\la_{r_k,\al_k}}{2}+|u_{r_k,\al_k}|^{q-2}+\be |u_{r_k,\al_k}|^{p-2}\right)u_{r_k,\al_k}.
\end{equation*}
In view of \eqref{2.21}  and \eqref{2-1}, we further have
$$-\De u_{r_k,\al_k}+\left(\|V\|_{\infty}+\frac{\la_{r_k,\al_k}}{2}\right)u_{r_k,\al_k}\ge 0$$ for large $k.$
Let $\theta_{r_k}$ be the principal eigenvalue of $-\De$  with Dirichlet boundary condition in $\Om_{r_k}$,
and $v_{r_k}>0$ be the corresponding normalized eigenfunction. Then $\theta_{r_k}=\frac{\theta_1}{r_k^2}$, and
$$\left(\frac{\theta_1}{r_k^2}+\|V\|_{\infty}+\frac{\la_{r_k,\al_k}}{2}\right)\int_{\Om_{r_k}}{u_{r_k,\al_k}v_{r_k}}dx\ge0.$$
Since  $\int_{\Om_{r_k}}{u_{r_k,\al_k}v_{r_k}}dx>0$, we have
$$\frac{\theta_1}{r_k^2}+\|V\|_{\infty}+\frac{\la_{r_k,\al_k}}{2}\ge 0,$$
which contradicts \eqref{2.21} for large $k.$
Hence the claim holds, that is, there exists $\al_1>0$ such that
\begin{equation}\label{2.43}
  g(\al)=\mathop{\liminf}\limits_{r\to\infty} \max_{\Om_r} u_{r,\al}>0
\end{equation}
for any $0<\al<\al_1$.

We consider $H^1(\Om_r)$ as a subspace of $H^1(\R^N)$ for any $r>0$. It follows from Lemma \ref{priori} that the set of solutions $\{u_{r,\al}:r>r_\al\}$ established in Theorem \ref{existbound} is bounded in $H^1(\R^N)$, so there exist $u_\al\in H^1(\R^N)$ and $\la_\al\in\R$ such that up to a subsequence:
$$\la_{r,\al}\to\la_\al,$$
$$u_{r,\al}\rightharpoonup u_\al\quad\text{in}\ H^1(\R^N),$$
$$u_{r,\al}\to u_\al\quad\text{in}\ L^k_{loc}(\R^N)\ \text{for all }2\le k< 2^*,$$
$$u_{r,\al}\to u_\al\quad\text{a.e.\;in}\ \R^N,$$
and $u_\al$ is a solution of the equation
\begin{equation}\label{2.17}
-\De u+Vu+\la_\al u=|u|^{q-2}u+\be|u|^{p-2}u\quad \text{in\ } \R^N.
\end{equation}
Therefore
\begin{equation}\label{2.5}
\int_{\R^N}|\nabla u_\al|^2dx+\int_{\R^N}Vu_\al^2dx+\la_\al\int_{\R^N}u_\al^2dx=\int_{\R^N}|u_\al|^qdx+\be\int_{\R^N}|u_\al|^pdx
\end{equation}
and the Pohozaev identity gives
\begin{equation}\label{2.6}
\begin{aligned}
&\frac{N-2}{2N}\int_{\R^N}|\nabla u_\al|^2dx+
 \frac{1}{2N}\int_{\R^N}\ti Vu_\al^2+\frac{1}{2}\int_{\R^N}Vu_\al^2dx+\frac{\la_\al}{2}\int_{\R^N}u_\al^2dx\\
 &\hspace{1cm} = \frac{1}{q}\int_{\R^N}|u_\al|^qdx+\frac{\be}{p}\int_{\R^N}|u_\al|^pdx.
\end{aligned}
\end{equation}
It follows from \eqref{2.5},  \eqref{2.6}, the H\"older inequality, the Gagliardo-Nirenberg inequality and $\be<0$ that
\begin{equation}
\begin{aligned}
&\left(\frac{1}{N}-\frac{\|\ti V_+\|_{\frac{N}{2}}S^{-1}}{2N}\right)\int_{\R^N}|\nabla u_\al|^2dx\\
&\hspace{1cm}\le \frac{C_{N,q}(q-2)}{2q}\left(\int_{\R^N} u_\al^2dx\right)^{\frac{2q-N(q-2)}{4}}
\left(\int_{\R^N}|\nabla u_\al|^2dx\right)^{\frac{N(q-2)}{4}}.
\end{aligned}
\end{equation}
As a consequence there holds for $u_\al\ne0$:
\begin{equation}\label{2.7}
\int_{\R^N}|\nabla u_\al|^2dx\ge\left(\frac{q(2-\|\ti V_+\|_{\frac{N}{2}}S^{-1})}{NC_{N,q}(q-2)}\right)^{\frac{4}{N(q-2)-4}}
\al^{\frac{q(N-2)-2N}{N(q-2)-4}}.
\end{equation}
Next it follows from \eqref{2.5}, \eqref{2.6}, \eqref{2.7}  and $2<p<2+\frac{4}{N}<2^*$  that
\begin{equation*}
\begin{aligned}
\frac{(2-q)\la_\al}{2q}\int_{\R^N}u_\al^2dx
 &=\;\frac{(N-2)q-2N}{2Nq}\int_{\R^N}|\nabla u_\al|^2dx
 +\frac{1}{2N}\int_{\R^N}\ti Vu_\al^2dx\\
 &\hspace{1cm}+\frac{q-2}{2q}\int_{\R^N}Vu_\al^2dx-\frac{(q-p)\be}{pq}\int_{\R^N}|u_\al|^pdx\\
 &\leq \frac{(N-2)q-2N}{2Nq}\int_{\R^N}|\nabla u_\al|^2dx+
 \frac{\|\ti V\|_\infty}{2N}\al+\frac{(q-2)\|V\|_\infty}{2q}\al\\
 &\hspace{1cm}-\frac{(q-p)\be C_{N,p}}{pq}\al^{\frac{2p-N(p-2)}{4}}
 \left(\int_{\R^N}|\nabla u_\al|^2dx\right)^{\frac{N(p-2)}{4}}\\
 &\to -\infty \quad \text{as}\ \al\to0.
 \end{aligned}
 \end{equation*}
Therefore, if $u_\al\ne0$ for $\al>0$ small there exists $\al_0>0$ such that $\la_\al>0$ for $0<\al<\al_0$.

In order to complete the proof we consider the case that there is a sequence $\al_k\to 0$ such that $u_{\al_k}=0$ for any $k$.
Assume without loss of generality that  $u_\al=0$ for any $\al\in (0,\al_1)$. Let $x_{r,\al}\in \Om_{r}$ be such that 
$u_{r,\al}(x_{r,\al})=\max\limits_{\Om_r}u_{r,\al}$. In view of \eqref{2.43}, there holds  $|x_{r,\al}|\to\infty$ as $r\to\infty.$
Otherwise, there exists $x_0\in\R$ such that, up to a subsequence $x_{r,\al}\to x_0$, and hence $u_\al(x_0)\ge d_\al>0.$ This contradicts $u_\al=0$.
We claim that dist$(x_{r,\al},\partial \Om_r)\to \infty$ as $r\to\infty$. Arguing by contradiction we assume that
$\mathop{\liminf}\limits_{r\to\infty}{\rm dist}(x_{r,\al},\partial \Om_r)=l<\infty$.
It follows from \eqref{2.43} that $l>0$. Let $w_r(x)=u_{r,\al}(x+x_{r,\al})$ for any $x\in \Sigma^r:=\{x\in\R^N:x+x_{r,\al}\in \Om_r\}$.
Then $w_r$ is bounded in $H^1(\R^N)$, and there is $w\in H^1(\R^N)$ such that $w_r\rightharpoonup w$ as $r\to\infty.$ By the regularity theory of elliptic partial equations and $\mathop{\liminf}\limits_{r\to\infty} u_{r,\al}(x_{r,\al})>d_\al>0,$ we infer that $w(0) \ge d_\al > 0$.
Assume without loss of the generality that, up to a subsequence,
$$\lim_{r\to\infty}\frac{x_{r,\al}}{|x_{r,\al}|}=e_1.$$
Setting
$$\Sigma=\{x\in\R^N:x\cdot e_1<l\}=\{x\in\R^N:x_1<l\},$$
we have $\phi(\cdot-x_{r,\al})\in C_c^\infty(\Om_r)$ for any $\phi\in C_c^\infty(\Sigma)$ and  $r$ large enough.
It then follows that
\begin{equation}\label{2-3}
\begin{aligned}
&\int_{\Om_r}\nabla u_{r,\al}\nabla \phi(\cdot-x_{r,\al})dx+\int_{\Om_r}V u_{r,\al}\phi(\cdot-x_{r,\al})dx+\la_{r,\al}\int_{\Om_r} u_{r,\al} \phi(\cdot-x_{r,\al})dx\\
&\hspace{1cm}=\int_{\Om_r} |u_{r,\al}|^{q-2}u_{r,\al} \phi(\cdot-x_{r,\al})dx+\be \int_{\Om_r} |u_{r,\al}|^{p-2}u_{r,\al} \phi(\cdot-x_{r,\al})dx.
\end{aligned}
\end{equation}
Since $|x_{r,\al}|\to\infty$ as $r\to\infty$ we have
\begin{equation}\label{2-18}
\begin{aligned}
\left|\int_{\Om_r}V u_{r,\al} \phi(\cdot-x_{r,\al})dx\right|&\le\int_{Supp \phi}\left|V(\cdot+x_{r,\al})w_r\phi\right| dx\\
&\le \|w_r\|_{2^*}\|\phi\|_{2^*}\left(\int_{Supp \phi}|V(\cdot+x_{r,\al})|^{\frac{N}{2}} dx\right)^{\frac{2}{N}}\\
&\le \|w_r\|_{2^*}\|\phi\|_{2^*}\left(\int_{\R^N\setminus B_{\frac{|x_{r,\al}|}{2}}}|V|^{\frac{N}{2}} dx\right)^{\frac{2}{N}}\\
&\to 0\quad \text{as\ } r\to\infty.
\end{aligned}
\end{equation}
Letting $r\to \infty$ in \eqref{2-3} we obtain for $\phi\in C_c^\infty(\Sigma)$:
\begin{equation*}
  \int_{\Sigma}\nabla w \cdot\nabla\phi dx+\la_\al\int_{\Sigma}w \phi dx
   = \int_{\Sigma} |w|^{q-2}w \phi dx+\be \int_{\Sigma} |w|^{p-2}w \phi dx.
\end{equation*}
Thus $w\in H_0^1(\Sigma)$ is a weak solution of the equation
 \begin{equation}\label{2.22}
-\De w+\la_\al w=|w|^{q-2}w+\be|w|^{p-2}w\quad \text{in}\ \Sigma.
\end{equation}
Hence we obtain a nontrivial nonnegative solution of \eqref{2.22} on a half space which is impossible by the Liouville theorem (see e.g.\ \cite{Esteban-Lions}). This proves that dist$(x_{r,\al},\partial \Om_r)\to \infty$ as $r\to\infty$. A similar argument as above shows that \eqref{2.22} holds for $\Sigma=\R^N$. Now we argue as in the case $u_\al\neq 0$ above that there exists $\al_2$ such that $\la_\al>0$ for any $0<\al<\al_2$.

Setting $\bar\al=\min\{\al_0,\al_1,\al_2\}$ the proof is complete.
\end{proof}

\vskip 0.1in
\begin{proof}[\bf Proof of Theorem \ref{boundexist}]
 The proof is an immediate consequence of Theorem~\ref{existbound} and Lemmas~\ref{blow} and \ref{multi}.
\end{proof}

\subsection{Proof of Theorem~\ref{boundexi}}
In this section we assume that the assumptions of Theorem \ref{boundexi} hold.
Let $0<\al<\al_V$ be fixed and consider the functions
$$h(t)=\frac{1}{2}\left(1-\|V_-\|_{\frac{N}{2}}S^{-1}\right)t^2-\frac{\be C_{N,p}}{p}\al^{\frac{2p-N(p-2)}{4}}
t^{\frac{N(p-2)}{2}}
-\frac{C_{N,q}}{q}\al^{\frac{2q-N(q-2)}{4}}t^{\frac{N(q-2)}{2}}$$
and
$$\phi(t):=\frac{1}{2}\left(1-\|V_-\|_{\frac{N}{2}}S^{-1}\right)t^{\frac{4-N(p-2)}{2}}
-\frac{C_{N,q}}{q}\al^{\frac{2q-N(q-2)}{4}}t^{\frac{N(q-p)}{2}}.$$
Note that $\phi$ admits a unique maximum at
$$\bar t=\left(\frac{q(4-N(p-2))(1-\|V_-\|_{\frac{N}{2}}S^{-1})}{2N(q-p)C_{N,q}}\right)^{\frac{2}{N(q-2)-4}}\al^{\frac{N(q-2)-2q}{2(N(q-2)-4)}}.$$
By a direct calculation and the definition of $\al_V$, we obtain
\begin{equation}\label{2.10}
\phi(\bar t)>\frac{\be C_{N,p}}{p}\al^{\frac{2p-N(p-2)}{4}}
\end{equation}
and $h(\bar t)>0$. In view of $2<p<2+\frac{4}{N}<q<2^*$ and \eqref{2.10}, there exist $0<R_1<T_\al<R_2$ such that
$h(t)<0$ for $0<t<R_1$ and for $t>R_2$,  $h(t)>0$ for $R_1<t<R_2$, and $h(T_\al)=\mathop{\max}\limits_{t\in\R^+}h(t)>0.$

\vskip 0.1in
Define
\begin{equation}\label{2.44}
\mathcal B_{r,\al}=\{u\in S_{r,\al}:\|\nabla u\|_2^2\le {T_\al^2}\}.
\end{equation}
Let $\theta$ be the principal eigenvalue of operator $-\De$ with Dirichlet boundary condition in $\Om$, and let $|\Om|$ be the volume of $\Om$.

\vskip 0.1in
\begin{thm}\label{2.25}
\begin{itemize}
\item[\rm (i)]  If $r<\frac{\sqrt{\theta\al}}{T_\al}$, then $\mathcal B_{r,\al}=\emptyset.$
\item[\rm (ii)] If
\begin{equation}\label{r}
r>\max\left\{\frac{\sqrt{\theta\al}}{T_\al},\left(\frac{2\be}{p}
\al^{\frac{p-2}{2}}\left(\theta\left(1+\|V\|_{\frac{N}{2}}S^{-1}\right)\right)^{-1}|\Om|^{\frac{2-p}{2}}\right)^{\frac{2}{N(p-2)-4}}\right\}
\end{equation}
then $\mathcal B_{r,\al}\neq\emptyset$ and
$$e_{r,\al}:=\inf_{u\in \mathcal B_{r,\al}}E_r(u)<0$$
is attained at some interior point $u_r>0$ of $\mathcal B_{r,\al}$. As a consequence, there exists a Lagrange multiplier $\la_r\in\R$ such that $(\la_r,u_r)$ is a solution of \eqref{boundeq1}. Moreover $\liminf\limits_{r\to\infty}\la_r>0$ holds true.
\end{itemize}
\end{thm}

\vskip 0.1in
\begin{proof}
Let $v_1\in S_{1,\al}$ be the positive normalized eigenfunction corresponding to $\theta$. The Poincar\'e inequality implies 
\begin{equation}
\int_{\Om_r}|\nabla u|^2dx\ge \frac{\theta\al}{r^2}
\end{equation}
for any $u\in S_{r,\al}.$ Since $T_\al$ is independent of $r$, there holds
$B_{r,\al}=\emptyset$  if and only if $r<\frac{\sqrt{\theta\al}}{T_\al}$.
Setting
\begin{equation}\label{2.12}
r_\al=\max\left\{\frac{\sqrt{\theta\al}}{T_\al},\left(\frac{p\theta\left(1+\|V\|_{\frac{N}{2}}S^{-1}\right)}{2\be}
\al^{\frac{2-p}{2}}|\Om|^{\frac{p-2}{2}}\right)^{\frac{2}{N(p-2)-4}}\right\}
\end{equation}
we construct for $r>r_\al$ a function $u_r\in S_{r,\al}$ such that $u_r\in\mathcal  B_{r,\al}$ and $E_r(u_r)<0$. Observe that
\begin{equation}\label{2.11}
\int_{\Om}|\nabla v_1|^2dx=\theta\al
\end{equation}
and
\begin{equation}
\begin{aligned}
\al=\int_{\Om}|v_1|^2dx\leq \left(\int_{\Om}|v_1|^pdx\right)^{\frac{2}{p}}|\Om|^{\frac{p-2}{p}}.
\end{aligned}
\end{equation}
We define $u_r\in S_{r,\al}$ by $u_r(x)=r^{-\frac{N}{2}}v_1(r^{-1}x)$ for $x\in \Om_r$. Then
$$\int_{\Om_r}|\nabla u_r|^2dx=r^{-2}\theta\al\quad\text{and}\quad
  \int_{\Om_r}| u_r|^pdx\ge r^{\frac{N(2-p)}{2}}\al^{\frac{p}{2}}|\Om|^{\frac{2-p}{2}}.$$
By \eqref{2.12}, $2<p<2+\frac{4}{N}$ and a direct calculation we have $u_r\in \mathcal B_{r,\al}$ and
\begin{equation}\label{2.13}
\begin{aligned}
E_{r}(u_r)
&= \frac{1}{2}\int_{\Om_r}|\nabla u_r|^2dx+\frac{1}{2}\int_{\Om_r}Vu_r^2dx-\frac{1}{q}\int_{\Om_r}|u_r|^qdx-\frac{\be}{p}\int_{\Om_r}|u_r|^pdx\\
&< \frac{1}{2}\left(1+\|V\|_{\frac{N}{2}}S^{-1}\right)r^{-2}\theta\al-\frac{\be}{p}r^{\frac{N(2-p)}{2}}\al^{\frac{p}{2}}|\Om|^{\frac{2-p}{2}}\\
&\le 0.
\end{aligned}
\end{equation}
It then follows from the Gagliardo-Nirenberg inequality that
\begin{equation}\label{2.18}
\begin{aligned}
E_{r}(u)&\ge\frac{1}{2}\left(1-\|V_-\|_{\frac{N}{2}}S^{-1}\right)\int_{\Om_r}|\nabla u|^2dx-\frac{C_{N,p}\be}{p}\al^{\frac{2p-N(p-2)}{4}}
\left(\int_{\Om_r}|\nabla u|^2dx\right)^{\frac{N(p-2)}{4}}\\
&\hspace{1cm} -\frac{C_{N,q}}{q}\al^{\frac{2q-N(q-2)}{4}}\left(\int_{\Om_r}|\nabla u|^2dx\right)^{\frac{N(q-2)}{4}}.
\end{aligned}
\end{equation}
As a consequence $E_{r}$ is bounded from below in $\mathcal B_{r,\al}$. By the Ekeland principle there exists a sequence $\{u_{n,r}\}\subset\mathcal  B_{r,\al}$ such that
$$E_{r}(u_{n,r})\to \inf_{u\in\mathcal  B_{r,\al}}E_r(u)E_{r}^\prime(u_{n,r})|_{T_{u_{n,r}}S_{r,\al}}\to 0
\quad \text{as}\ n\to\infty.$$
Consequently there exists $u_r\in H_0^1(\Om_r)$ such that
$$u_{n,r}\rightharpoonup u_r\ \text{in}\ H_0^1(\Om_r) \quad\text{and}\quad u_{n,r}\to u_r\ \text{in}\ L^k(\Om_r)\ \text{for all } 2\le k< 2^*.$$
Moreover,
$$\|\nabla u_r\|_2^2\le \liminf_{n\to\infty}\|\nabla u_{n,r}\|_2^2\le T_\al^2,$$
that is, $u_r\in\mathcal B_{r,\al}$. Note that
$$\int_{\Om_r}Vu_{n,r}^2dx\to\int_{\Om_r}Vu_{r}^2dx\quad\text{as}\ n\to\infty$$
hence
\begin{equation}
e_{r,\al}\le E_{r}(u_r)\le \liminf_{n\to\infty}E_{r}(u_{n,r})=e_{r,\al}.
\end{equation}
It follows that $u_{n,r}\to u_r $ in $H_0^1(\Om_r)$, so $E_{r}(u_r)<0$. Therefore $u_r$ is an interior point of $\mathcal B_{r,\al}$ because 
$E_{r}(u)\ge h(T_\al)>0$ for any $u\in\partial\mathcal B_{r,\al}$ by \eqref{2.18}. The Lagrange multiplier theorem implies that there exists $\la_r\in\R$ such that $(\la_r,u_r)$ is a solution of \eqref{boundeq1}. Moreover,
\begin{equation}\label{2.14}
\begin{aligned}
\la_r\al&=\int_{\Om_r}|u_r|^qdx+\be\int_{\Om_r}|u_r|^pdx-\int_{\Om_r}|\nabla u_r|^2dx-\int_{\Om_r}V u_r^2dx\\
&=\frac{(p-2)\be}{p}\int_{\Om_r}|u_r|^pdx+\frac{q-2}{q}\int_{\Om_r}|u_r|^qdx-2E_r(u_r)\\
&>-2E_{r}(u_r)=-2e_{r,\al}.
\end{aligned}
\end{equation}
It follows from the definition of $e_{r,\al}$ that  $e_{r,\al}$ is nonincreasing with respect to $r$. Hence,
$e_{r,\al}\le  e_{r_\al,\al}<0$ for any $r>r_\al$ and $0<\al<\al_V.$
In view of \eqref{2.14}, we have $$\liminf_{r\to\infty}\la_r>0.$$
Finally, the strong maximum principle implies $u_r>0$.
\end{proof}

\vskip 0.1in
\begin{proof}[\bf Proof of Theorem \ref{boundexi}]
The proof is a direct consequence of Theorem \ref{2.25}, Lemma \ref{blow} and Remark \ref{2.26}.
\end{proof}

\subsection{Proof of Theorem~\ref{multiplicity}}
In this subsection we assume that the assumptions of Theorem \ref{multiplicity} hold. For $s\in \left[\frac{1}{2},1\right]$ we define the functional
$\mathcal E_{r,s}:S_{r,\al}\to\R$ by
\begin{equation}\label{2-6}
{\mathcal E_{r,s}(u)}=\frac{1}{2}\int_{\Om_r}|\nabla u|^2dx+\frac{1}{2}\int_{\Om_r}Vu^2dx-s
\left(\frac{\be}{p}\int_{\Om_r}|u|^pdx+\frac{1}{q}\int_{\Om_r}|u|^qdx\right).
\end{equation}
Note that if $u\in S_{r,\al}$ is a critical point of $\mathcal E_{r,s}$ then there exists $\la\in \R$ such that $(\la,u)$ is a solution of the problem
\begin{equation}\label{meq11}
\begin{cases}
-\De u+Vu+\la u=s|u|^{q-2}u+s\be |u|^{p-2}u\quad \text{in}\ \Om_r,\\
u\in H^1_0(\Om_r),\quad\int_{\Om_r}|u|^2dx=\al.
\end{cases}
\end{equation}

\vskip 0.1in

\begin{lemma}\label{moun1}
For $0<\al<\ti \al_V$ where $\ti \al_V$ is defined in Theorem \ref{multiplicity}, there exist  $\ti r_\al>0$ and $u^0,u^1\in S_{r_\al,\al}$ such that
\begin{itemize}
\item [\rm (i)]
For $r>\ti r_\al$ and $s\in \left[\frac{1}{2},1\right]$ we have $\mathcal E_{r,s}(u^1)\le0$ and
$$\mathcal E_{r,s}(u^0)<\frac{N(q-2)-4}{4}
\left(\frac{2(1-\|V_-\|_\frac{N}{2}S)}{N(q-2)}\right)^{\frac{N(q-2)}{N(q-2)-4}}A^{\frac{4}{4-N(q-2)}}\al^{\frac{N(q-2)-2q}{N(q-2)-4}}$$
where $$A=\left(\frac{C_{N,q}(q-2)(N(q-2)-4)}{q (p-2)(4-N(p-2))}+\frac{C_{N,q}}{q}\right).$$
Moreover, 
$$\|\nabla u^0\|_2^2< \left(\frac{2(1-\|V_-\|_{\frac{N}{2}}S^{-1})}{N(q-2)A}\right)^{\frac{4}{N(q-2)-4}}\al^{\frac{N(q-2)-2q}{N(q-2)-4}}$$
and
$$\|\nabla u^1\|_2^2> \left(\frac{2(1-\|V_-\|_{\frac{N}{2}}S^{-1})}{N(q-2)A}\right)^{\frac{4}{N(q-2)-4}}\al^{\frac{N(q-2)-2q}{N(q-2)-4}}$$

\item[\rm (ii)]If  $u\in S_{r,\al}$ satisfies
 $$\|\nabla u\|_2^2=\left(\frac{2(1-\|V_-\|_{\frac{N}{2}}S^{-1})}{N(q-2)A}\right)^{\frac{4}{N(q-2)-4}}
 \al^{\frac{N(q-2)-2q}{N(q-2)-4}},$$
 then there holds
  $$\mathcal E_{r,s}(u)\ge \frac{N(q-2)-4}{4}
\left(\frac{2(1-\|V_-\|_\frac{N}{2}S)}{N(q-2)}\right)^{\frac{N(q-2)}{N(q-2)-4}}A^{\frac{4}{4-N(q-2)}}\al^{\frac{N(q-2)-2q}{N(q-2)-4}}.$$

\item [\rm(iii)] Let
$${m_{r,s}(\al)}=\inf_{\gamma\in \Gamma_{r,\al}}\sup_{t\in [0,1]}\mathcal E_{r,s}(\gamma(t)),$$
where
 {$$\Gamma_{r,\al}=\left\{\gamma\in C([0,1],S_{r,\al}):\gamma(0)=u^0, \gamma(1)=u^1\right\}.$$}
Then 
$$m_{r,s}(\al)\ge\frac{N(q-2)-4}{4}
\left(\frac{2(1-\|V_-\|_\frac{N}{2}S)}{N(q-2)}\right)^{\frac{N(q-2)}{N(q-2)-4}}A^{\frac{4}{4-N(q-2)}}\al^{\frac{N(q-2)-2q}{N(q-2)-4}}$$
and
$$
\begin{aligned}
 m_{r,s}(\al)&\le \frac{N(q-2)-4}{2}\left(\frac{\theta (1+\|V\|_{\frac{N}{2}}S^{-1})}{N(q-2)}\right)^{\frac{N(q-2)}{N(q-2)-4}}\\
                  &\hspace{1cm} \times (4q)^{\frac{4}{N(q-2)-4}}|\Om|^{\frac{2(q-2)}{N(q-2)-4}}\al^{\frac{N(q-2)-2q}{N(q-2)-4}},
\end{aligned}
$$
where $\theta$ is the principal eigenvalue of $-\De$ with Dirichlet boundary condition in $\Om$.
\end{itemize}
\end{lemma}

\vskip 0.1in
\begin{proof}
Let $v_1\in S_{1,\al}$ be the positive normalized eigenfunction of $-\De$ with Dirichlet boundary condition in $\Om$ associated to $\theta$. Then we have
\begin{equation}
\int_{\Om}|\nabla v_1|^2dx=\theta\al \quad\text{and}\quad \int_{\Om}| v_1|^pdx\ge\al^{\frac{p}{2}}|\Om|^{\frac{2-p}{2}}.
\end{equation}
Setting $v_t(x)=t^{\frac{N}{2}}v_1(tx)$ for $x\in B_{\frac{1}{t}}$ and $t>0$ we get
\begin{equation}\label{2.151}
\begin{aligned}
\mathcal E_{\frac{1}{t},s}(v_t)\leq&\; \frac{1}{2}\left(1+\|V\|_{\frac{N}{2}}S^{-1}\right)t^2\int_{\Om}|\nabla v_1|^2dx
-\frac{\be}{2p}t^{\frac{N(p-2)}{2}}\int_{\Om}| v_1|^pdx\\
&\;-\frac{1}{2q}t^{\frac{N(q-2)}{2}}\int_{\Om}| v_1|^qdx\\
\le &\;h(t),\\
\end{aligned}
\end{equation}
where the function $h:\R^+\to\R$ is defined by
$$h(t)=\frac{1}{2}\left(1+\|V\|_{\frac{N}{2}}S^{-1}\right)t^2\theta\al-\frac{1}{2q}\al^{\frac{q}{2}}|\Om|^{\frac{2-q}{2}}t^{\frac{N(q-2)}{2}}.$$
A simple computation shows that $h(t_0)=0$ for
$$t_0 := \left((1+\|V\|_{\frac{N}{2}}S^{-1})q\theta\al^{\frac{2-q}{2}}|\Om|^{\frac{q-2}{2}}\right)^{\frac{2}{N(q-2)-4}}$$
and $h(t)<0$ for any $t>t_0$, $h(t)>0$ for any $0<t<t_0$. Moreover $h$ achieves its maximum at
$$t_\al=\left(\frac{4q(1+\|V\|_{\frac{N}{2}}S^{-1})\theta}{N(q-2)}\al^{\frac{2-q}{2}}|\Om|^{\frac{q-2}{2}}\right)^{\frac{2}{N(q-2)-4}}.$$
This implies
\begin{equation}\label{2.301}
\mathcal E_{r,s}(v_{t_0})= \mathcal E_{\frac{1}{t_0},s}(v_{t_0})\le h(t_0)=0
\end{equation}
for any {$r\ge\frac{1}{t_0}$} and $s\in\left[\frac{1}{2},1\right]$.
There exists $0<t_1<t_\al$ such that for any $t\in [0,t_1],$
\begin{equation}\label{2.311}
h(t)<\frac{N(q-2)-4}{4}
\left(\frac{2(1-\|V_-\|_\frac{N}{2}S)}{N(q-2)}\right)^{\frac{N(q-2)}{N(q-2)-4}}A^{\frac{4}{4-N(q-2)}}\al^{\frac{N(q-2)-2q}{N(q-2)-4}}.
\end{equation}
On the other hand, it follows from the  Gagliardo-Nirenberg inequality and the
H\"older inequality that
\begin{equation}\label{2.231}
\begin{aligned}
\mathcal E_{r,s}(u)&\ge \frac{1}{2}\left(1-\|V_-\|_{\frac{N}{2}}S^{-1}\right)\int_{\Om_r}|\nabla u|^2dx-
\frac{C_{N,q}\al^{\frac{2q-N(q-2)}{4}}}{q}\left(\int_{\Om_r}|\nabla u|^2dx\right)^{\frac{N(q-2)}{4}}\\
&\hspace{1cm} -\frac{C_{N,p}\be\al^{\frac{2p-N(p-2)}{4}}}{p}\left(\int_{\Om_r}|\nabla u|^2dx\right)^{\frac{N(p-2)}{4}}.
\end{aligned}
\end{equation}
Let $f:\R^+\to \R$ be defined by
\begin{equation}\label{lower}
f(t)=\frac{1}{2}\left(1-\|V_-\|_{\frac{N}{2}}S^{-1}\right)t-\frac{\be C_{N,p}}{p}\al^{\frac{2p-N(p-2)}{4}}t^{\frac{N(p-2)}{4}}
-\frac{C_{N,q}}{q}\al^{\frac{2q-N(q-2)}{4}}t^{\frac{N(q-2)}{4}}.
\end{equation}
In view of $2<p<2+\frac{2}{N}<q<2^*$ and the definition of $\ti \al_V$, there exist $0<l_1<l_M<l_2$ such that
$h(t)<0$ for any $0<t<l_1$ and $t>l_2$,  $f(t)>0$ for $l_1<t<l_2$ and $f(l_M)=\mathop{\max}\limits_{t\in\R^+}f(t)>0.$
Let
$$t_2=\left(\frac{\be q C_{N,p}(p-2)(4-N(p-2))}{p C_{N,q}(q-2)(N(q-2)-4)}\right)^{\frac{4}{N(q-p)}}\al^{\frac{N-2}{N}}.$$
Then by a direct calculation, we have $f^{\prime\prime}(t)\le 0$ if and only if $t\ge t_2.$ Hence
$$\max_{t\in\R^+}f(t)=\max_{t\in[t_2,\infty)}f(t).$$
Note that for any $t\ge t_2,$
\begin{equation}\label{2.32}
\begin{aligned}
f(t) &= \frac{1}{2}\left(1-\|V_-\|_{\frac{N}{2}}S^{-1}\right)t-\frac{\be C_{N,p}}{p}\al^{\frac{(q-p)(N-2)}{4}}\al^{\frac{2q-N(q-2)}{4}}
t^{\frac{N(p-2)}{4}}-\frac{C_{N,q}}{q}\al^{\frac{2q-N(q-2)}{4}}t^{\frac{N(q-2)}{4}}\\
&\ge \frac{1}{2}\left(1-\|V_-\|_{\frac{N}{2}}S^{-1}\right)t-\left(\frac{C_{N,q}(q-2)(N(q-2)-4)}{q (p-2)(4-N(p-2))}+\frac{C_{N,q}}{q}\right)\al^{\frac{2q-N(q-2)}{4}}t^{\frac{N(q-2)}{4}}\\
&=:g(t).
\end{aligned}
\end{equation}
Let
$$A=\left(\frac{C_{N,q}(q-2)(N(q-2)-4)}{q (p-2)(4-N(p-2))}+\frac{C_{N,q}}{q}\right).$$
and
$$t_g=\left(\frac{2(1-\|V_-\|_{\frac{N}{2}}S^{-1})}{N(q-2)A}\right)^{\frac{4}{N(q-2)-4}}\al^{\frac{N(q-2)-2q}{N(q-2)-4}}$$
so that $t_g>t_2$ by the definition of $\ti \al_V$, $\mathop{\max}\limits_{t\in [t_2,\infty)}g(t)=g(t_g)$ and
\begin{equation}
\begin{aligned}
\max_{t\in\R^+}f(t)&\ge\max_{t\in [t_2,\infty)}g(t)\\
&=\frac{(N(q-2)-4)}{4}
\left(\frac{2(1-\|V_-\|_\frac{N}{2}S)}{N(q-2)}\right)^{\frac{N(q-2)}{N(q-2)-4}}A^{\frac{4}{4-N(q-2)}}\al^{\frac{N(q-2)-2q}{N(q-2)-4}}.
\end{aligned}
\end{equation}
Set
 $\bar r_\al=\max\left\{\frac{1}{t_1},\sqrt{\frac{2\theta\al}{t_g}}\right\},$ then $v_{\frac{1}{\bar r_\al}}\in S_{r,\al}$
 for any $r>\bar r_\al$, and
 \begin{equation}\label{2.281}
\|\nabla v_{\frac{1}{\bar r_\al}}\|_2^2=\left(\frac{1}{\bar r_\al}\right)^2\|\nabla v_1\|_2^2<\left(\frac{2(1-\|V_-\|_{\frac{N}{2}}S^{-1})}{N(q-2)A}\right)^{\frac{4}{N(q-2)-4}}\al^{\frac{N(q-2)-2q}{N(q-2)-4}},
\end{equation}
moreover,
\begin{equation}\label{2.291}
\mathcal E_{\bar r_\al,s}(v_{\frac{1}{\bar r_\al}})\le h\left(\frac{1}{\bar r_\al}\right)\le h(t_1).
\end{equation}
Let $u^0=v_{\frac{1}{\bar r_\al}},u^1=v_{t_0}$ and
\begin{equation}\label{2.41}
{{\ti r_\al}=\max\left\{\frac{1}{t_0}, \bar r_\al\right\}}.
\end{equation}
Then the statement (i) holds by \eqref{2.281}, \eqref{2.291}, \eqref{2.311} and \eqref{2.301}.
\vskip 0.1in
(ii) holds by \eqref{2.32} and a direct calculation.
\vskip 0.1in
(iii) In view of $\mathcal E_{r,s}(u^1)\le0$ for any $\gamma\in\Gamma_{r,\al}$ and the definition of $t_0$, we have
 $$\|\nabla \gamma(0)\|_2^2<t_g< \|\nabla \gamma(1)\|_2^2.$$
 It then follows from \eqref{2.32} that
 $$
 \begin{aligned}
 \max_{t\in[0,1]}\mathcal E_{r,s}(\gamma(t)) &\ge g(t_g)\\
    &= \frac{N(q-2)-4}{4}
          \left(\frac{2(1-\|V_-\|_\frac{N}{2}S^{-1})}{N(q-2)}\right)^{\frac{N(q-2)}{N(q-2)-4}}A^{\frac{4}{4-N(q-2)}}\al^{\frac{N(q-2)-2q}{N(q-2)-4}}
\end{aligned}
$$
for any $\gamma\in\Gamma_{r,\al}$, hence  the first inequality in (iii) holds.
We define a path $\ga:[0,1]\to S_{r,\al}$ by
 \begin{equation}\label{2.161}
 \gamma(t):\Om_r\to\R,\quad 
   x\mapsto \left(\tau t_0+(1-\tau)\frac{1}{\ti r_\al}\right)^{\frac{N}{2}}v_1\left(\left(\tau t_0+(1-\tau)\frac{1}{\ti r_\al}\right) x\right).
\end{equation}
Then $\gamma\in \Gamma_{r,\al}$, and the second inequality in (iii) follows from \eqref{2.151}.
\end{proof}

\begin{thm}\label{existmeq1}
Assume $0<\al<\ti \al_V$ where $\ti \al_V$ is given in Theorem \ref{multiplicity}.
Let $r>\ti r_\al$, where $\ti r_\al$ is defined in Lemma \ref{moun1}. Then problem \eqref{meq11} admits a solution $(\la_{r,s},u_{r,s})$ for almost every $s\in\left[\frac{1}{2},1\right]$. Moreover, there hold $u_{r,s}> 0$ and $\mathcal E_{r,s}(u_{r,s})=m_{r,s}(\al)$.
\end{thm}

\begin{proof}
The proof is similar to that of Theorem \ref{existmeq} by Lemma \ref{moun1} and Theorem \ref{criticalprinciple}. We omit it here.
\end{proof}

\begin{lemma}\label{priori1}
For fixed $\al>0$ the set of solutions $u\in S_{r,\al}$ of \eqref{meq11} is bounded uniformly in $s$ and $r$.
\end{lemma}

\begin{proof}
By an argument similar to that in Lemma \ref{priori} we have
\begin{equation*}
 \begin{aligned}
 &\frac{1}{N}\int_{\Om_r}|\nabla u|^2dx-\frac{1}{2N}\int_{\partial \Om_r}|\nabla u|^2(x\cdot \mathbf{n})d\sigma-
 \frac{1}{2N}\int_{\Om_r}\ti Vu^2dx\\
 &\hspace{1cm}= \frac{(q-2)s}{2q}\int_{ \Om_r}|u|^qdx+\frac{\be(p-2)s}{2p}\int_{ \Om_r}|u|^pdx\\
 &\hspace{1cm}= \frac{q-2}{2}\left(\frac{1}{2}\int_{\Om_r}|\nabla u|^2dx+\frac{1}{2}\int_{ \Om_r}Vu^2dx-m_{r,s}(\al)\right)
 +\frac{\be s(p-q)}{2p}\int_{ \Om_r}|u|^pdx.
 \end{aligned}
 \end{equation*}
 It then follows from  the Gagliardo-Nirenberg inequality that
 \begin{align*}
 &\frac{q-2}{2}m_{r,s}(\al)+\frac{\be (q-p)C_{N,p}}{2p}\al^{\frac{2p-N(p-2)}{4}}\left(\int_{\Om_r}|\nabla u|^2dx\right)^{\frac{N(p-2)}{4}}\\
&\hspace{1cm} \ge \frac{q-2}{2}m_{r,s}(\al)+\frac{\be s(q-p)}{2p}\int_{ \Om_r}|u|^pdx.\\
&\hspace{1cm} \ge \frac{q-2}{2}\left(\frac{1}{2}\int_{\Om_r}|\nabla u|^2dx+\frac{1}{2}\int_{ \Om_r}Vu^2dx\right)\\
&\hspace{2cm} -\left(\frac{1}{N}\int_{\Om_r}|\nabla u|^2dx-\frac{1}{2N}\int_{\partial \Om_r}|\nabla u|^2(x\cdot \mathbf{n})d\sigma-
 \frac{1}{2N}\int_{\Om_r}(\nabla V\cdot x)u^2dx\right)\\
&\hspace{1cm} \ge \frac{N(q-2)-4}{4N}\int_{\Om_r}|\nabla u|^2dx-\al\left(\frac{1}{2N}\|\nabla V\cdot x\|_\infty+\frac{q-2}{4}\|V\|_\infty\right).
 \end{align*}
As a consequence of $2<p<2+\frac{4}{N}$ and (iii) in Lemma \ref{moun1} we can bound $\int_{\Om_r}|\nabla u|^2dx$ uniformly in $s$ and $r$.
\end{proof}

\begin{thm}\label{existbound1}
Assume $0<\al<\ti \al_V$, where $\ti \al_V$ is given in Theorem \ref{multiplicity}, and let $r>\ti r_\al$, where $\ti r_\al$ is defined in Lemma \ref{moun1}.
Then the following hold:
\begin{itemize}
\item [\rm (i)] Equation \eqref{boundeq1} admits a solution $(\la_{r,\al},u_{r,\al})$ for every $r>\ti r_\al$ such that $u_{r,\al}>0$ in $\Om_r$.
\item [\rm (ii)]
There is $0<\bar\al\le \ti \al_V$  such that
 $$\mathop{\liminf}\limits_{r\to\infty}\la_{r,\al}>0\quad
\text{for any}\ 0<\al<\bar\al.$$
\end{itemize}
\end{thm}

\begin{proof}
The proof of (i) is similar to that of Theorem \ref{existbound}, so we omit it. As be consider $H^1_0(\Om_r)$ as a subspace of $H^1(\R^N)$ 
for every $r>0$. In view of Lemma \ref{priori1}, there are  $\la_\al$ and $u_\al\in H^1(\R^N)$ such that up to a subsequence, 
$$u_{r,\al}\rightharpoonup u_\al \quad \text{in}\ H^1(\R^N) \quad\text{and}\quad
      \lim _{r\to\infty}\la_{r,\al}\to \la_\al.$$
ARguing by contradiciton we assume that $\la_{\al_n}\le 0$ for some sequence$\al_n\to0$. 
Let $\theta_r$ be the principal eigenvalue of $-\De$ with Dirichlet boundary condtion in $\Om_r$ and let $v_>0r$ be the corresponding normalized eigenfunction. Testing \eqref{boundeq1} with $v_r$ we obtain
$$(\theta_r+\la_{r,\al_n})\int_{\Om_r}u_{r,{\al_n}}v_rdx+\int_{\Om_r}Vu_{r,{\al_n}}v_rdx\ge 0.$$
In view of $\int_{\Om_r}u_{r,{\al_n}}v_rdx>0$ and $\theta_r=r^{-2}\theta_1$, there holds
$$\max_{x\in\R^N}V+\la_{r,\al_n}+r^{-2}\theta_1\ge0.$$
Hence there exists $C>0$ independent of $n$ such that $|\la_{\al_n}|\le C$ for any $n$.

{\bf Case 1.} There is subsequence denoted still by $\{\al_n\}$ such that $u_{\al_n}=0.$
We first claim that there exists $d_n>0$ for any $n$ such that
\begin{equation}\label{2.34}
\liminf_{r\to\infty}\sup_{z\in \R^N}\int_{B(z,1)}u_{r,\al_n}^2dx\ge d_n.
\end{equation}
Otherwise, the concentration compactness principle implies for every $n$ that
$$u_{r,\al_n}\to 0\quad\text{in}\ L^t(\R^N)\quad\text{as}\ r\to\infty,\quad \text{for all } 2<t<2^*.$$
By  the diagonal principle, \eqref{boundeq1} and $|\la_{r,\al_n}|\le 2C$ for large $r$, there exists $r_n\to\infty$ such that
$$\int_{ \Om_r}|\nabla u_{r_n,\al_n}|^2dx\le C$$
for some $C$ independent of $n$, contradicting (iii) in Lemma \ref{moun1} for large $n$.
As a consequence \eqref{2.34} holds, and there is $z_{r,\al_n}\in \Om_r$ with $|z_{r,\al_n}|\to\infty$ such that
$$\int_{B(z_{r,\al_n},1)}u_{r,\al_n}^2dx\ge\frac{d_n}{2}.$$
Moreover, ${\rm dist}(z_{r,\al_n},\partial \Om_r) \to \infty$ as $r\to \infty$ by an argument similar to that in Lemma \ref{multi}.
Now, for $n$ fixed let $v_r(x)=u_{r,\al_n}(x+z_{r,\al_n})$ for $x\in \Sigma^r:=\left\{x\in\R^N: x+z_{r,\al_n}\in \Om_r\right\}$.
It follows from Lemma \ref{priori1} that there is $v\in H^1(\R^N)$ with $v\neq0$ such that $v_r\rightharpoonup v$.
Observe that for every $\phi\in C_c^\infty(\R^N)$ there is $r$ large such that $\phi(\cdot-z_{r,\al_n})\in C_c^\infty(\Om_r)$
due to ${\rm dist}(z_{r,\al_n},\partial\Om_r)\to\infty$ as $r\to\infty$. It follows that
\begin{equation}\label{2-3-1}
\begin{aligned}
&\int_{\Om_r}\nabla u_{r,\al_n}\nabla \phi(\cdot-z_{r,\al_n})dx+\int_{\Om_r}V u_{r,\al_n} \phi(\cdot-z_{r,\al_n})dx+\la_{r,\al_n}\int_{\Om_r} u_{r,\al_n} \phi(\cdot-z_{r,\al_n})dx\\
&\hspace{1cm}= \int_{\Om_r} |u_{r,\al_n}|^{q-2}u_{r,\al_n} \phi(\cdot-z_{r,\al_n})dx
    + \be \int_{\Om_r} |u_{r,\al_n}|^{p-2}u_{r,\al_n} \phi(\cdot-z_{r,\al_n})dx.
\end{aligned}
\end{equation}
From $|z_{r,\al_n}|\to\infty$ as $r\to\infty$ we deduce
\begin{equation*}
\begin{aligned}
\left|\int_{\Om_r}V u_{r,\al_n} \phi(\cdot-z_{r,\al_n})dx\right|&\le\int_{Supp \phi}\left|V(\cdot+z_{r,\al_n})v_r\phi\right|dx\\
&\le \|v_r\|_{2^*}\|\phi\|_{2^*}\left(\int_{\R^N\setminus B_{\frac{|z_{r,\al_n}|}{2}}}|V|^{\frac{N}{2}} dx\right)^{\frac{2}{N}}\\
&\to 0\quad \text{as\ } r\to\infty.
\end{aligned}
\end{equation*}
{Letting $r\to\infty$  in \eqref{2-3-1} we get for every $\phi\in C_c^\infty(\R^N)$:
\begin{equation*}
\int_{\R^N}\nabla v \cdot\nabla\phi dx+\la_{\al_n}\int_{\R^N}v \phi dx
  = \int_{\R^N} |v|^{q-2}v \phi dx+\be \int_{\R^N} |v|^{p-2}v \phi dx 
\end{equation*}
Therefore $v\in H^1(\R^N)$ is a weak solution} of the equation
$$-\De v+\la_{\al_n} v=\be|v|^{p-2}v+|v|^{q-2}v\quad\text{in}\ \R^N$$
and
$$\int_{\R^N}|\nabla v|^2dx+\la_{\al_n}\int_{\R^N}|v|^2dx=
\be\int_{\R^N}|v|^pdx+\int_{\R^N}|v|^qdx.$$
The Pohozaev identity implies
$$\frac{N-2}{2N}\int_{\R^N}|\nabla v|^2dx+\frac{\la_{\al_n}}{2}\int_{\R^N}|v|^2dx=
\frac{\be}{p}\int_{\R^N}|v|^pdx+\frac{1}{q}\int_{\R^N}|v|^qdx$$
hence
\begin{equation}\label{2.40}
\frac{\la_{\al_n}}{N}\int_{\R^N}|v|^2dx=\frac{\be(2N-p(N-2))}{2Np}\int_{\R^N}|v|^pdx+
\frac{2N-q(N-2)}{{2Nq}}\int_{\R^N}|v|^qdx.
\end{equation}
As a result we have $\la_{\al_n}>0,$ which is a contradiction.

{\bf Case 2.}  $u_{\al_n}\neq0$ for $n$ large.  Note that $u_{\al_n}$ satisfies
\begin{equation}\label{2.35}
-\De u_{\al_n} +Vu_{\al_n}+\la_{\al_n} u_{\al_n}=\be|u_{\al_n}|^{p-2}u_{\al_n}+|u_{\al_n}|^{q-2}u_{\al_n}.
\end{equation}
If $v_{r,{\al_n}} := u_{r,{\al_n}}-u_{\al_n}$ satisfies
\begin{equation}\label{2.39}
  \limsup_{r\to\infty}\max_{z\in \R^N}\int_{B(z,1)}v_{r,{\al_n}}^2dx=0,
\end{equation}
then the concentration compactness principle implies $u_{r,{\al_n}}\to u_{\al_n}$ in $L^t(\R^N)$ for any $2<t<2^*$.
It then follows from \eqref{boundeq1} and \eqref{2.35} that
\begin{equation*}
\begin{aligned}
\int_{\Om_r}|\nabla u_{r,{\al_n}}|^2dx+{\al_n}\la_{r,\al_n}&=\be\int_{\Om_r}|u_{r,{\al_n}}|^pdx
+\int_{\Om_r}|u_{r,{\al_n}}|^qdx-\int_{\Om_r}Vu_{r,{\al_n}}^2dx\\
&=\be\int_{\R^N}|u_{r,{\al_n}}|^pdx+\int_{\R^N}|u_{r,{\al_n}}|^qdx-\int_{\R^N}Vu_{r,{\al_n}}^2dx\\
&\to \be\int_{\R^N}|u_{{\al_n}}|^pdx+\int_{\R^N}|u_{{\al_n}}|^qdx-\int_{\R^N}Vu_{{\al_n}}^2dx\\
&=\int_{\R^N}|\nabla u_{{\al_n}}|^2dx+\la_{\al_n}\int_{\R^N}u_{\al_n}^2dx.
\end{aligned}
\end{equation*}
Using $\la_{r,\al_n}\to \la_{\al_n}$ as $r\to\infty$ we further have
\begin{equation}\label{2.36}
\int_{\Om_r}|\nabla u_{r,{\al_n}}|^2dx+{\al_n}\la_{\al_n}\to \int_{\R^N}|\nabla u_{{\al_n}}|^2dx
+\la_{\al_n}\int_{\R^N}u_{\al_n}^2dx\quad {as}\ r\to\infty.
\end{equation}
In view of \eqref{2.36}, (iii) in Lemma \ref{moun1} and  $|\la_{\al_n}|\le C$ for large $n$, there holds
\begin{equation}\label{2.37}
\int_{\R^N}|\nabla u_{{\al_n}}|^2dx\to \infty\quad {as}\ n\to\infty.
\end{equation}
By \eqref{2.35} and the Pohozaev identity, we obtain
\begin{equation}\label{2.38}
\begin{aligned}
0\le&\;\frac{(2-q)\la_{\al_n}}{2q}\int_{\R^N}u_{\al_n}^2dx\\
=&\;\frac{(N-2)q-2N}{2Nq}\int_{\R^N}|\nabla u_{\al_n}|^2dx+
\frac{1}{2N}\int_{\R^N}\ti Vu_{\al_n}^2\\
&\;+\frac{q-2}{2q}\int_{\R^N}Vu_{\al_n}^2dx-\frac{(q-p)\be}{pq}\int_{\R^N}|u_{\al_n}|^pdx\\
\leq&\; \frac{(N-2)q-2N}{2Nq}\int_{\R^N}|\nabla u_{\al_n}|^2dx+
\frac{\|\ti V\|_\infty}{2N}{\al_n}+\frac{(q-2)\|V\|_\infty}{2q}{\al_n}\\
\to&\;-\infty ,\quad \text{as}\ n\to\infty.
\end{aligned}
\end{equation}
Therefore \eqref{2.39} cannot occur.
Consequently there exist  $d_n>0$ and $z_{r,\al_n}\in \Om_r$ with $|z_{r,\al_n}|\to\infty$ as $r\to\infty$  such that
$$\int_{B(z_{r,\al_n},1)}v_{r,{\al_n}}^2dx>{d_n}.$$
Then $\ti v_{r,{\al_n}} := v_{r,{\al_n}}(\cdot+z_{r,\al_n})\rightharpoonup \ti v_{\al_n}\ne0$, and $\ti v_{\al_n}$ is a nonnegative solution of
$$-\De v+\la_{\al_n} v=\be|v|^{p-2}v+|v|^{q-2}v\quad\text{in}\ \R^N.$$
In fact, we have $ \mathop{\liminf}\limits_{r\to\infty}{\rm dist}(z_{r,\al_n},\partial \Om_r)=\infty$ by the Liouville theorem on the half space.
It follows from an argument similar to that of \eqref{2.40} that $\la_{\al_n}>0$ for large $n$,  which is a contradiction.
\end{proof}

\begin{proof}[\bf Proof of Theorem \ref{multiplicity}] The proof is a direct result of Theorem \ref{existbound1},
Lemma \ref{blow} and Remark \ref{2.26}. 
\end{proof}

\vskip 0.1in
\section{Normalized solutions in $\R^N$}
In this section we specialize $\Om$ to  the unit ball $B_1\subset\R^N$ in Theorems \ref{boundexist}, \ref{boundexi} and \ref{multiplicity}, 
and then study the existence of the solutions of \eqref{eq} in $\R^N$ by analyzing the compactness of normalized solutions $u_r$ in the ball $B_r$
as the radius $r$ tends to infinity.

\vskip 0.1in
As in Section 2 we consider $H_0^1(B_r)$ as a subspace of $H^1(\R^N)$ for any $r>0$.

\begin{lemma}\label{H1bounded}
Under the assumptions of Theorem \ref{boundexist} (resp.\ Theorem \ref{boundexi} and Theorem \ref{multiplicity}),
let $(\la_r,u_r)$, $r>0$, be a family of solutions of \eqref{boundeq1} with $\Om_r=B_r$ as established in
Theorem \ref{boundexist} (resp.  Theorem \ref{boundexi} and  Theorem \ref{multiplicity}). Then $u_r$ is bounded uniformly in $r$ and $\liminf\limits_{r\to\infty}\la_{r} >0$.
\end{lemma}

\begin{proof}
If $(\la_r,u_r)$ is a solution of \eqref{boundeq1} as  established in Theorem \ref{boundexist}, then the lemma is a direct consequence of the Lemmas~\ref{priori} and \ref{multi}.

If $(\la_r,u_r)$ is a solution of \eqref{boundeq1} as established in
Theorem \ref{boundexi}, then the lemma follows from $u_r\in \mathcal B_{r,\al}$, \eqref{2.44} and Theorem \ref{2.25}.

Finally, if $(\la_r,u_r)$ is a solution of \eqref{boundeq1} as established in
Theorems \ref{multiplicity}, then the lemma is a consequence of Lemma~\ref{priori1} and (ii) of Theorem~\ref{existbound1}.
\end{proof}

The next lemma is concerned with the profile decomposition of normalized solutions.

\begin{lemma}\label{profile}
Under the assumptions of Theorem \ref{boundexist} (resp.\ Theorem \ref{boundexi} and Theorem \ref{multiplicity}),
let $\{(\la_r,u_r)\}$ be a sequence of solutions of \eqref{boundeq1} with $\Om_r=B_r$ as established in
Theorem \ref{boundexist} (resp.\ Theorem \ref{boundexi} and Theorem \ref{multiplicity}).
Then there exists a subsequence, still denoted by $\{(\la_r, u_r)\}$, and there exist $\la>0$,
$l\in \mathbb N\cup\{0\}$, $u_0\in H^1(\R^N),$  $w^1,\cdots,w^l\in H^1(\R^N)\setminus\{0\},$
$\{z_r^k\}\subset\R^N$ with $z_r^k\in B_r$, such that
$$|z_r^k|\to\infty,\quad{\rm dist}(z_r^k,\partial B_r)\to\infty,\quad |z_r^k-z_r^{k^\prime}|\to\infty$$
for any $k,k^\prime=1,\cdots,l$ and $k\neq k^\prime$. These satisfy:
\begin{itemize}
\item [\rm (i)] $\la_r\to\la\quad\text{as}\ r\to\infty.$
\item[\rm(ii)] $\left\|u_r-u_0-\sum\limits_{k=1}^lw^k(\cdot-z_r^k)\right\|_{H^1}\to 0$ as $r\to\infty$.
\item[\rm (iii)] $u_0$ is a solution of
$$-\De u+Vu+\la u=|u|^{q-2}u+\be|u|^{p-2}u\quad\text{in}\ \R^N.$$
\item[\rm(iv)] $w^k$, $k=1,\dots,l$ is a solution of
\begin{equation}\label{3.13}
-\De u+\la u=|u|^{q-2}u+\be|u|^{p-2}u\quad\text{in}\ \R^N.
\end{equation}
\end{itemize}
\end{lemma}

\begin{proof}
The argument is quite standard. In view of Lemma \ref{H1bounded}, there exist $u_0\in H^1(\R^N)$ and $\la>0$ such that, up to a subsequence,
$$\la_r\to\la,\quad u_r\rightharpoonup u_0\ \ \text{in } H^1(\R^N).$$
Moreover, $u_0$ is a solution of
$$-\De u+Vu+\la u=|u|^{q-2}u+\be|u|^{p-2}u\quad \text{in}\ \R^N.$$
Set $v_r^1=u_r-u_0,$ then $v_r^1\rightharpoonup 0$ as $r\to\infty$ in $H^1(\R^N)$. We consider two cases.
\vskip 0.1in
{\bf Case 1.} If
$$\limsup_{r\to\infty}\sup_{z\in\R^N}\int_{B(z,1)}|v_r^1|^2dx\to 0,$$
then we claim that $v_r^1\to0$ in $H^1(\R^N).$ In fact, it follows from the concentration compactness
principle that $v_r^1\to0$ in $L^t(\R^N)$ for any $2<t<2^*$. By the Brezis-Lieb lemma and the Vitali convergence theorem we have
\begin{align*}
o_r(1)&= \int_{\R^N}|v_r^1|^qdx+\be\int_{\R^N}|v_r^1|^pdx\\
&= \int_{\R^N}|u_r|^qdx+\be\int_{\R^N}|u_r|^pdx-\int_{\R^N}|u_0|^qdx-\be\int_{\R^N}|u_0|^pdx+o_r(1)\\
&= \int_{\R^N}|\nabla u_r|^2dx+\int_{\R^N}V u_r^2dx +\la\int_{\R^N}|u_r|^2dx\\
&\hspace{1cm} -\int_{\R^N}|\nabla u_0|^2dx-\int_{\R^N}V u_0^2dx -\la\int_{\R^N}|u_0|^2dx+o_r(1)\\
&= \int_{\R^N}|\nabla (u_r- u_0)|^2dx+\la\int_{\R^N}|u_r-u_0|^2dx+o_r(1).
\end{align*}
Thus $v_r^1=u_r- u_0 \to 0$ as claimed.

\vskip 0.1in
{\bf Case 2.} If there exists $z_r^1\in B_r$ such that
$$\int_{B(z_r^1,1)}|v_r^1|^2dx>d>0,$$
then $|z_r^1|\to\infty$ as $r\to\infty.$ Otherwise,  there is $R>0$ such that $|z_r^1|<R$. Hence
\begin{equation}\label{3.1}
\int_{B(0,R+1)}|v_r^1|^2dx\ge\int_{B(z_r^1,1)}|v_r^1|^2dx >d.
\end{equation}
This contradicts $v_r^1\rightharpoonup0$ as $r\to\infty.$
It follows from arguments similar to those in the proof of Theorem \ref{multi}
that dist$(z_r^1,\partial B_r)\to \infty$ as $r\to \infty$. Then
$w_r^1=v_r^1(\cdot+z_r^1)\rightharpoonup w^1$ in $H^1(\R^N)$, and a calculation similar to \eqref{2-18} shows that $w^1$ satisfies
$$-\De w^1+\la w^1=|w^1|^{q-2}w^1+\be|w^1|^{p-2}w^1\quad \text{in}\ \R^N.$$

Setting $v_r^2(x)=v_r^1(x)-w^1(x-z_r^1)$ we repeat the above process. If the process has been repeated $m$ times then
\begin{equation}
\begin{aligned}
0&\le \lim_{r\to\infty}\|u_r-u_0-\sum_{k=1}^{m-1}w^k(\cdot-z_r^k)\|_{H^1}\\
&=\lim_{r\to\infty} \|u_r\|_{H^1}-\|u_0\|_{H^1}-\sum_{k=1}^{m-1}\|w^k\|_{H^1}
\end{aligned}
\end{equation}
and thus
\begin{equation}\label{3.9}
\lim_{r\to\infty} \|u_r\|_{H^1}\ge\|u_0\|_{H^1}+\sum_{k=1}^{m-1}\|w^k\|_{H^1}.
\end{equation}
We claim that there is $d>0$ such that $\|w\|_{H^1}\ge d$ for any solution $w$ with $\|w\|_2^2\le \al$ of \eqref{3.13}.
Once this has been proved the theorem follows immediately.
\vskip 0.1in
Consider first the case $\be\le0$. If $w$ satisfies $\|w\|_2^2\le \al$ and
$$-\De w +\la w=|w|^{q-2}w+\be|w|^{p-2}w \quad \text{in}\ \R^N$$
then
$$\int_{\R^N}|\nabla w|^2dx+\la\int_{\R^N}|w|^2dx=\int_{\R^N}|w|^qdx+\be\int_{\R^N}|w|^pdx.$$
The Pohozaev identity implies 
$$\frac{N-2}{2N}\int_{\R^N}|\nabla w|^2dx+\frac{\la}{2}\int_{\R^N}|w|^2dx=\frac{1}{q}\int_{\R^N}|w|^qdx
+\frac{\be}{p}\int_{\R^N}|w|^pdx.$$
It follows that
$$\frac{1}{N}\int_{\R^N}|\nabla w|^2dx=\frac{q-2}{2q}\int_{\R^N}|w|^qdx+\frac{\be(p-2)}{2p}\int_{\R^N}|w|^pdx.$$
Since $\be\le0$ the Gagliardo-Nirenberg inequality now yields the required bound of $w$ from below:
$$\left(\int_{\R^N}|\nabla w|^2dx\right)^{\frac{N(q-2)-4}{4}}\ge\frac{2q}{N(q-2)}\left(\int_{\R^N}| w|^2dx\right)^{\frac{N(q-2)-2q}{4}}\ge \frac{2q}{N(q-2)}\al^{\frac{N(q-2)-2q}{4}}.$$
\vskip 0.1in
It remains to consider the case $\be>0$. Recall that the ground state energy of the fixed frequency equation \eqref{3.13} is positive as is the mountain pass value. This implies
$$\frac{1}{2}\int_{\R^N}|\nabla w|^2dx+\frac{\la}{2}\int_{\R^N}|w|^2dx-\frac{1}{q}\int_{\R^N}|w|^qdx
-\frac{\be}{p}\int_{\R^N}|w|^pdx > 0.$$
It follows from the Pohozaev identity that
$$\frac{1}{N}\int_{\R^N}|\nabla w|^2dx > 0.$$
Therefore $w$ is also bounded from below in the case $\be>0$.
\end{proof}

Now we consider the existence of a normalized solution in $\R^N$.

\begin{proof}[\bf Proofs of Theorems \ref{exis}, \ref{multipli} and \ref{exist}]
Let $\{(\la_r,u_r)\}$ be a family of solutions of \eqref{boundeq1} with $\Om_r=B_r$ as established in
Theorems \ref{boundexist} (resp.\ in Theorem \ref{boundexi} or Theorem \ref{multiplicity}).
In view of Lemma \ref{profile}, if $l=0$ then $u_r\to u_0$ as $r\to\infty$ in $H^1(\R^N)$, and we are done.
If $l>0$ then, up to a subsequence, there exists $i\in\{1,2,\cdots,l\}$ such that $|z_r^i|\le\min\{|z_r^j|:j=1,\cdots,l\}$. Assume without loss of the
generality that $i=1.$ Let
$$\frac{|z_r^k-z_r^1|}{|z_r^1|}\to d_k\in[0,\infty],$$
and set $\nu^*={\frac{1}{4}}\min\{1,\rho,d^*\}$, where $d^*=\min\{d_k:d_k\neq0\}$. Fixing $0<\delta<\nu^*$ and consider the annulus
$$A_r=B\left(z_r^1,\frac{3}{2}\delta|z_r^1|\right)\setminus B\left(z_r^1,\frac{1}{2}\delta|z_r^1|\right).$$
We claim that dist$(z_r^i, A_r)\ge{\frac{1}{4}}\delta|z_r^1|$ for any $i=2,\cdots,l$ and large $r$.
In fact, if $d_i=0$ for some $i=2,\cdots,l$, then $z_r^i\in B\left(z_r^1,\frac{1}{4}\delta|z_r^1|\right)$ and the claim holds.
If $d_i\ge d^*>0$ for some $i=2,\cdots,l$, then
$|z_r^i-z_r^1|>\frac{d_i}{2}|z_r^1|>2\delta|z_r^1|$. Then dist$(z_r^i,A_r)>\frac{1}{2}\delta|z_r^1|$
and the claim holds.
It follows from  Lemma \ref{profile} that $\|u_r\|_{L^p(A_r)}\to 0$.
Note that $u_r$ satisfies
\begin{equation}\label{3.12}
-\De u_r+Vu_r +\la u_r\le|u_r|^{q-2}u_r+\be|u_r|^{p-2}u_r \quad \text{in}\ \R^N
\end{equation}
for large $r$, where $\la:=\frac{1}{2}\mathop{\liminf}\limits_{r\to\infty}\la_r>0$ due to Lemma \ref{H1bounded}.
Setting
$$R_m=\overline{B\left(z_r^1,\frac{3}{2}\delta|z_r^1|-m\right)}\setminus B\left(z_r^1,\frac{1}{2}\delta|z_r^1|+m\right)
    \quad \text{for } m\in\mathbb N^+$$
the interior estimates of elliptic partial differential equations (see e.g.\ \cite{Gilbarg-Trudinger}) imply
\begin{equation}\label{3.2}
\|u_r\|^{q-2}_{L^\infty(R_1)} < \frac{\la}{8}\quad\text{and}\quad |\be|\cdot\|u_r\|^{p-2}_{L^\infty(R_1)} < \frac{\la}{8}.
\end{equation}
Let $\xi_m\in C_c^\infty(\R)$ be a cut-off function with $0\le\xi_m(t)\le 1,$ $|\xi_m^\prime(t)|<4$ for any $t\in \R$, and
\begin{equation*}
\xi_m(t)=
\begin{cases}
1\quad &\text{if } \frac{1}{2}\delta|z_r^1|+m<t<\frac{3}{2}\delta|z_r^1|-m,\\
0 &\text{if } t<\frac{1}{2}\delta|z_r^1|+m-1\  \text{or}\ t>\frac{3}{2}\delta|z_r^1|-m+1.
\end{cases}
\end{equation*}
Now we set $\phi_m(x)=\xi(|x-z_r^1|)$ and test \eqref{3.12} with $\phi^2_mu_r$:
\begin{equation}\label{3.3}
\begin{aligned}
&\int_{R_{m-1}}|\nabla u_r|^2\phi_m^2dx+\int_{R_{m-1}}Vu_r^2\phi_m^2dx+\la\int_{R_{m-1}}u_r^2\phi_m^2dx -\int_{R_{m-1}}|u_r|^q\phi_m^2dx\\
&\hspace{2cm}-\be\int_{R_{m-1}}|u_r|^p\phi_m^2dx\\
&\hspace{1cm}\leq -2\int_{R_{m-1}}u_r\phi_m\nabla u_r\cdot\nabla \phi_mdx\\
&\hspace{1cm}\leq 8\int_{R_{m-1}\setminus R_m}|u_r\|\nabla u_r|dx.
\end{aligned}
\end{equation}
Since $|z_r^1|\to\infty$ and {$\liminf_{|x|\to\infty}V(x)\ge0$} there exists $\bar r$
such that $V(x)\ge-\frac{\la}{4}$ for any $x\in A_r$ with $r>\bar r$. It follows from  \eqref{3.2} that
\begin{equation*}
\begin{aligned}
&\int_{R_{m-1}}|\nabla u_r|^2\phi_m^2dx+\int_{R_{m-1}}Vu_r^2\phi_m^2dx+\la\int_{R_{m-1}}u_r^2\phi_m^2dx - \int_{R_{m-1}}|u_r|^q\phi_m^2dx\\
&\hspace{2cm} -\be\int_{R_{m-1}}|u_r|^p\phi_m^2dx\\
&\hspace{1cm}\ge \int_{R_{m-1}}|\nabla u_r|^2\phi_m^2dx+\frac{\la}{2}\int_{R_{m-1}}u_r^2\phi_m^2dx\\
&\hspace{1cm}\ge \min\left\{1,\frac{\la}{2}\right\}\left(\int_{R_{m-1}}|\nabla u_r|^2\phi_m^2dx+\int_{R_{m-1}}u_r^2\phi_m^2dx\right).
\end{aligned}
\end{equation*}
Now \eqref{3.3} implies
\begin{equation}\label{3.10}
\int_{R_{m}}|\nabla u_r|^2dx+\int_{R_{m}}u_r^2dx\le \kappa\left(\int_{R_{m-1}\setminus R_{m}}|\nabla u_r|^2dx+\int_{R_{m-1}\setminus R_m}u_r^2dx\right)
\end{equation}
where $\kappa=4\max\left\{1,\frac2{\la}\right\}$. Setting
$$a_m=\int_{R_{m}}|\nabla u_r|^2dx+\int_{R_{m}}u_r^2dx \quad\text{and}\quad \theta=\frac{\kappa}{1+\kappa}$$
we obtain from \eqref{3.10}
$$a_m\le\theta a_{m-1}\le\theta^m\max_{r}\|u_r\|_{H^1}=\left(\max_{r}\|u_r\|_{H^1}\right)e^{m\ln\theta}=:Ae^{m\ln\theta}.$$
Observe that
$$D_r=B(z_r^1,\delta|z_r^1|+1)\setminus B(z_r^1,\delta|z_r^1|-1) \subset R_M $$
with
$$M=\left[\frac{\delta|z_r^1|}{2}\right]-1.$$
Therefore we have
\begin{equation*}
\int_{D_r}|\nabla u_r|^2dx+\int_{D_r}u_r^2dx\le Ae^{M\ln\theta}\le Ae^{\frac{\delta|z_r^1|}{4}\ln\theta}.
\end{equation*}
It follows from  the interior estimate that
\begin{equation}\label{3.4}
|\nabla u_r|+|u_r| \le Ae^{-c|z_r^1|}
\end{equation}
for any $x$ with $\delta|z_r^1|-\frac{1}{2}<|x-z_r^1|<\delta|z_r^1|+\frac{1}{2}$ and large $r$, where the constants $A>0$ and $c>0$ are independent of $r$.

Let $\Gamma_1=\partial B(z_r^1,\delta |z_r^1|)\cap B_r$ and $\Gamma_2=B(z_r^1,\delta |z_r^1|)\cap\partial B_r$.
Multiplying both sides of \eqref{boundeq1} with  $z_r^1\cdot\nabla u_r$  and integrating, we obtain
\begin{equation}\label{3.5}
\begin{aligned}
&\frac{1}{2}\int_{B(z_r^1,\delta |z_r^1|)\cap B_r}(z_r^1\cdot\nabla V)u_r^2dx\\
&\hspace{1cm}= \frac{1}{2}\int_{\Gamma_1\cup \Gamma_2}(z_r^1\cdot \mathbf n) |\nabla u_r^1|^2d\sigma
-\int_{\Gamma_1\cup \Gamma_2}\left(\nabla u_r\cdot \mathbf n\right)(z_r^1\cdot \nabla u_r)d\sigma\\
&\hspace{2cm} -\int_{\Gamma_1\cup \Gamma_2}(z_r^1\cdot \mathbf n)\left(\frac{Vu_r^2}{2}-\frac{\be|u_r|^p}{p}-\frac{|u_r|^q}{q}+\frac{\la}{2}|u_r|^2\right)d\sigma\\
&\hspace{1cm}=: \mathcal A_1+ \mathcal A_2
\end{aligned}
\end{equation}
where $\mathbf n$ denotes the outward unit normal vector on $\partial\left(B(z_r^1,\delta |z_r^1|)\cap B_r\right)$ and
\begin{equation}
\begin{aligned}
\mathcal A_i &:= \frac{1}{2}\int_{\Gamma_i}(z_r^1\cdot\mathbf n) |\nabla u_r^1|^2d\sigma
-\int_{\Gamma_i}\left(\nabla u_r\cdot \mathbf n\right)(z_r^1\cdot \nabla u_r)d\sigma\\
&\hspace{1cm} -\int_{\Gamma_i}(z_r^1\cdot \mathbf n)\left(\frac{Vu_r^2}{2}-\frac{\be|u_r|^p}{p}-\frac{|u_r|^q}{q}
     +\frac{\la}{2}|u_r|^2\right)d\sigma
\end{aligned}
\end{equation}
for $i=1,2$. If $\Gamma_2\neq\emptyset$, then $z_r^1\cdot \mathbf n(x)>0$ for any $x\in \Gamma_2$ by the choice of $\delta$. Note that
$\mathbf n(x)=\frac{x}{|x|}=\frac{\nabla u_r}{|\nabla u_r|}$ for any $x\in\Gamma_2$, hence
\begin{equation}\label{3.6}
\mathcal A_2=-\frac{1}{2}\int_{ \Gamma_2}(z_r^1\cdot\mathbf n) |\nabla u_r|^2d\sigma\le 0.
\end{equation}
It follows from \eqref{3.4} that there is $\tau_0>0$ independent of $r$ such that
\begin{equation}\label{3.7}
\limsup_{r\to\infty}e^{\tau_0|z_r^1|}\mathcal A_1=0.
\end{equation}
In view of Lemma \ref{profile} we have
$$\liminf_{r\to\infty}\int_{B(z_r^1,\delta|z_r^1|)\cap B_r}u_r^2dx\ge \|w^1\|_2^2>0.$$
Moreover, there holds
\begin{equation}\label{3.8}
\begin{aligned}
&\lim_{r\to\infty}e^{\tau_0|z_r^1|}\int_{B(z_r^1,\delta |z_r^1|)\cap B_r}(z_r^1\cdot\nabla V)u_r^2dx\\
&\hspace{1cm}
   \ge \lim_{r\to\infty}e^{\tau_0|z_r^1|}\left(\inf_{x\in B(z_r^1,\delta |z_r^1|)}(z_r^1\cdot\nabla V(x))\right)
          \int_{B(z_r^1,\delta|z_r^1|)\cap B_r}u_r^2dx\\
&\hspace{1cm}> 0.
\end{aligned}
\end{equation}
Now we obtain a contradiction using \eqref{3.5}, \eqref{3.6}, \eqref{3.7} and \eqref{3.8}. It follows that $l=0$ and we conclude that
$u_r\to u_0$ as $r\to \infty$ in $H^1(\R^N).$ Consequently, $u_0$ is a normalized solution of \eqref{eq}.
\end{proof}

\end{document}